\documentclass{amsart}
\usepackage{etex}
\usepackage{enumitem}

\usepackage{tikz}
\usetikzlibrary{decorations.markings}
\usepackage{float}

\usepackage{amsmath}
\usepackage{amssymb}
\usepackage[all]{xy}

\usepackage{eepic}
\usepackage{xcolor}
\usepackage{comment}

\usepackage{hyperref}

\newtheorem{thm}{Theorem}[section]
\newtheorem{lemma}[thm]{Lemma}
\newtheorem{cor}[thm]{Corollary}

\theoremstyle{definition}
\newtheorem{defn}[thm]{Definition}

\newtheorem{rem}[thm]{Remark}

\newtheorem{fact}[thm]{Fact}

\newcommand{\acts}{\curvearrowright}

\newcommand{\res}{\restriction}

\newcommand{\bd}{\boldsymbol{\Delta}}
\newcommand{\bs}{\boldsymbol{\Sigma}}
\newcommand{\bp}{\boldsymbol{\Pi}}
\newcommand{\bbP}{\mathbb{P}}

\newcommand{\Z}{\mathbb{Z}}

\newcommand{\N}{\mathbb{N}}

\newcommand{\dom}{\text{\rm dom}}

\newcommand{\ocl}[1]{\overline{#1}}

\newcommand{\sR}{\mathcal{R}}

\newcommand{\xd}{{\dot{x}}}
\newcommand{\forces}{\Vdash}
\newcommand{\ac}{\mathsf{AC}}
\newcommand{\zf}{\mathsf{ZF}}
\newcommand{\sat}{\text{sat}}

\begin{document}

\title{Forcing Constructions and Countable Borel Equivalence Relations}

\author{Su Gao}
\address{Department of Mathematics\\ University of North Texas\\ 1155 Union Circle \#311430\\  Denton, TX 76203\\ USA}
\email{sgao@unt.edu}
\author{Steve Jackson}
\address{Department of Mathematics\\ University of North Texas\\ 1155 Union Circle \#311430\\  Denton, TX 76203\\ USA}
\email{jackson@unt.edu}
\author{Edward Krohne}
\address{Department of Mathematics\\ University of North Texas\\ 1155 Union Circle \#311430\\  Denton, TX 76203\\ USA}
\email{EdwardKrhone@my.unt.edu}
\author{Brandon Seward}
\address{Department of Mathematics\\ University of Michigan\\ 2074 East Hall\\ 530 Church Street\\  Ann Arbor, MI 48109 \\ USA}
\email{b.m.seward@gmail.com}
\date{\today}
\thanks{}

\begin{abstract} We prove a number of results about countable Borel equivalence relations with forcing constructions and arguments. These results reveal hidden regularity properties of Borel complete sections on certain orbits. As consequences they imply the nonexistence of Borel complete sections with certain features.
\end{abstract}

\maketitle \thispagestyle{empty}


\section{Introduction}

This paper is a contribution to the study of countable Borel
equivalence relations. We consider Borel actions of countable
groups on Polish spaces and study the orbit equivalence relations
which they generate. Properties such as hyperfiniteness,
treeability, chromatic numbers, matchings, etc. have received
much interest both in ergodic theory and descriptive set theory.
Typically, investigations into these properties begin with the
construction of Borel complete sections possessing special properties.
In this paper we introduce new methods based on forcing techniques
for studying Borel complete sections. We use forcing constructions
to prove the existence of certain regularity phenomena in complete sections.
This of course prevents the existence of complete sections
with certain features. We remark that our work here is entirely in the Borel
setting, as our results generally fail if null sets are ignored.

Recall that a set $S$ is a \emph{complete section} for an
equivalence relation $E$ if $S$ meets every $E$-class. A classic result
on complete sections is the Slaman--Steel lemma which states that
every aperiodic countable Borel equivalence relation $E$ admits a
decreasing sequence of Borel complete sections $S_n$ with empty
intersection (this result is stated explicitly as Lemma 6.7 of \cite{KechrisMiller},
where they attribute it to Slaman-Steel; the proof is implicit in 
Lemma~1 of \cite{SlamanSteel}). This result played an important role in their proof
that every equivalence relation generated by a Borel action of $\Z$ is
hyperfinite. A long standing open problem asks if every equivalence
relation generated by a Borel action of a countable amenable group
must be hyperfinite, and progress on this problem is in some ways connected to
strengthening the Slaman--Steel lemma. 
In particular, constructing sequences of complete sections (``marker sets'') with 
certain geometric properties is central to the proofs of 
 \cite{GaoJackson,SchneiderSeward} that every equivalence relation
 generated by the Borel action of an abelian, or even locally nilpotent,
group is hyperfinite. In particular, the 
constructions in \cite{GaoJackson,SchneiderSeward} build complete
sections $B_n$ (facial boundaries) which are sequentially orthogonal, or repel
one another, so that the sequence $S_n = \bigcup_{i > n} B_i$
is decreasing and vanishes. Thus, the question of 
what kinds of marker sets various equivalence relations can admit is 
an important one.

Our first theorem unveils a curious property which
limits how quickly a sequence of complete sections can
vanish. In fact, this theorem says that if a sequence of
complete sections vanishes, then it must do so
arbitrarily slowly.

\begin{thm} \label{main1}
Let $\Gamma$ be a countable group, $X$ a compact Polish space, and
$\Gamma \acts X$ a continuous action giving rise to the orbit
equivalence relation $E$.  Let $(S_n)_{n \in \N}$ be a sequence of Borel
complete sections of $E$. If $(A_n)_{n \in \N}$ is any sequence of
finite subsets of $\Gamma$ such that every finite subset of $\Gamma$
is contained in some $A_n$, then there is an $x \in X$ such that for
infinitely many $n$ we have $A_n\cdot x \cap S_n \neq \varnothing$.
\end{thm}

We remark that the above result is easily seen to be inherited from
subspaces, so one can instead simply require that $X$ contain a
compact invariant subset. In particular, by results in \cite{GJS1, GJS2} the
above result holds when $X = F(2^\Gamma)$, where $F(2^\Gamma)$ is the set of points
in $2^\Gamma$ having trivial stabilizer.

This theorem was motivated by a similar result for the case when each
$S_n$ is clopen, the proof of which is a straightforward topological
argument without forcing. We will define a forcing notion, called
\emph{orbit forcing}, that will allow us to give a proof of
Theorem~\ref{main1} that is essentially a generalization of the
topological proof. It will turn out that forcing can be removed and a
pure topological proof is possible (in fact we will give such a
proof), but the forcing proof is shorter and more intuitive.

We remark that, after learning of the above theorem, C. Conley and A. Marks
obtained another interesting result on the behavior of distances
to sequences of complete sections \cite{CMa}.

The orbit forcing can be used to obtain more results, the following
being an example.

\begin{thm}\label{main2}
If $B \subseteq F(2^\Gamma)$ is a Borel complete section then $B$
meets some orbit recurrently, i.e., there is $x\in F(2^\Gamma)$ and
finite $T\subseteq \Gamma$ such that for any $y\in[x]$, $T\cdot y\cap
B\neq\varnothing$.
\end{thm}

Again, if $B$ is  assumed to be clopen then the result follows
from the fact that minimal elements form a dense set in
$F(2^\Gamma)$ \cite[Theorem 5.3.6]{GJS2}. We find that the most direct way to obtain this ``Borel
result'' is to mimic the topological proof but use forcing.

The above result can be strengthened in various ways. For example, in
the case of $\Gamma = \Z^d$ we can require that the recurrences occur
at odd distances (distance here refers to the taxi-cab metric
induced by the $\ell_1$ norm $\| (g_1,\dots,g_d)\|=|g_1|+\cdots +|g_d|$).

\begin{thm}\label{main3} Let $d\geq 1$.
If $B \subseteq F(2^{\Z^d})$ is a Borel complete section then $B$
meets some orbit recurrently with odd distances, i.e., there is
$x\in F(2^{\Z^d})$ and finite $T\subseteq \{ g \in \Z^d \colon \| g\| \text{ is odd }\}$
such that for any $y\in[x]$, $T\cdot y\cap B\neq\varnothing$.
\end{thm}

It is worth noting that Theorem \ref{main3}, and in fact all of the
forcing results in this paper, can be proved using the orbit forcing method.
However, we believe that forcing arguments in general
may provide a new path for studying countable Borel equivalence
relations. Thus, in order to demonstrate the flexibility of forcing
arguments in this setting, we define and use other forcing notions beyond the
orbit forcing. We choose to prove the above theorem by using a notion of an
odd minimal $2$-coloring forcing.

Another forcing notion we introduce is that of a grid periodicity forcing.
Using this forcing, we obtain the following result which reveals a
surprising amount of regularity in complete sections.

\begin{thm}\label{main4}
Let $d\geq 1$. If $B \subseteq F(2^{\Z^d})$ is a Borel complete
section then there is an $x\in F(2^{\Z^d})$ and a lattice $L\subseteq
\Z^d$ such that $L\cdot x\subseteq B$.
\end{thm}

If $B \subseteq F(2^{\Z^d})$ is a Borel set but not a complete section,
then there is $x$ with $\Z^d \cdot x \cap B = \varnothing$. Thus we
have the following immediate corollary.

\begin{cor}\label{main5}
Let $d\geq 1$. If $B \subseteq F(2^{\Z^d})$ is Borel then there is an
$x\in F(2^{\Z^d})$ and a lattice $L\subseteq \Z^d$ such that either
$L\cdot x\subseteq B$ or $L\cdot x\cap B=\varnothing$.
\end{cor}

A. Marks \cite{Marks} has proved a similar result for free groups using
Borel determinacy. Also, after discussing Theorem \ref{main4} with him,
he generalized Theorem \ref{main4} to all countable residually finite groups \cite{Ma}. His proof
also uses forcing, though it uses none of the forcing notions we
introduce here. This again suggests that the flexibility in choosing
a forcing notion may be important for future applications to Borel
equivalence relations.

The above results can be viewed as ruling out certain Borel complete
sections (marker sets) with strong regularity
properties. Alternatively, they can be viewed as 
saying that Borel marker sets must, on some equivalence classes, 
exhibit  regular structure.
In general,  regular marker sets and structures are
desirable in hyperfiniteness proofs or Borel combinatorial results
(e.g., in the study of Borel chromatic numbers). The negative results
stated below unveil a fine line between what is possible and what is
not possible.

In \cite{GaoJackson}, the first two authors proved that all equivalence
relations generated by Borel actions of countable abelian groups are
hyperfinite (this has since been extended to locally nilpotent groups
\cite{SchneiderSeward}). For the finite equivalence relations they construct,
the shapes of the classes at a sufficiently large scale look like rectangles.
However, at finer and finer scales the shapes appear to be increasingly fractal-like.
We use forcing to prove a
claim made in \cite{GaoJackson} stating that this fractal-like behavior
is necessary. This fact indicates that
hyperfiniteness results of this type have a necessary degree of complexity. The
theorem below is stated for rectangles but the proof works for any
reasonable polygon.

\begin{thm} \label{main8}
There does not exist a sequence $\sR_n$ of Borel finite subequivalence
relations on $F(2^{\Z^2})$ satisfying all the following:

\begin{enumerate}
\item[\rm (1)] {\em (regular shape)} For each $n$, each marker region
$R$ of $\sR_n$ is a rectangle.

\item[\rm (2)] {\em (bounded size)} 
For each $n$, there is an upper bound $w(n)$ on the size
of the edge lengths of the marker regions $R$ in $\sR_n$.

\item[\rm (3)] {\em (increasing size)} Letting $v(n)$ denote the
smallest edge length of a marker region $R$ of $\sR_n$, we have $\lim_n
v(n)=+\infty$.

\item[\rm (4)] {\rm (vanishing boundary)} 
For each $x \in F(2^{\Z^2})$ we have that $\lim_n \rho(x, \partial \sR_n)=+\infty$.
\end{enumerate}
\end{thm}

Our last two negative results touch upon the theory of
Borel chromatic numbers. It is not difficult to show that $F(2^{\Z^2})$
has Borel chromatic number strictly greater
than $2$. By using the odd minimal $2$-coloring forcing, we show
that in fact there cannot exist any Borel chromatic coloring
of $F(2^{\Z^2})$ which uses two colors on arbitrarily large regions.

\begin{thm} \label{main9}
There does not exist a Borel chromatic coloring $f \colon F(2^{\Z^2}) \to
\{0,1,\dots,k\}$ such that for all $x \in F(2^{\Z^2})$ there are
arbitrarily large rectangles $R$ in $\Z^2$ such that $f(R\cdot x)$
consists of only two elements of $\{0,1,\dots, k\}$.
\end{thm}

A useful structure for the study of Borel graphs and chromatic numbers is the
notion of \emph{toast} or ``barrier'' as named in \cite{CMi}. For
example, in \cite{CMi} C. Conley and B. Miller used barriers to
prove that for a large class of Borel graphs $G$, the Baire-measurable
and $\mu$-measurable chromatic numbers of $G$ are at most
twice the standard chromatic number of $G$ minus one. In a similar fashion,
the existence of a toast structure on $F(2^{\Z^2})$ would easily imply
the existence of a Borel chromatic $3$-coloring. As a consequence
of Theorem \ref{main1}, we deduce that there is no toast structure
which is layered.

\begin{cor}\label{main7}
There is no Borel layered toast on $F(2^{\Z^d})$, i.e., there is no
sequence of finite subequivalence relations $\{T_n\}$ of $E_{\Z^d}$ on
some subsets $\dom(T_n)\subseteq F(2^{\Z^d})$ such that
\begin{enumerate}
\item[\rm (0)] $\bigcup_n \dom(T_n)= F(2^{\Z^d})$;
\item[\rm (1)] For each $T_n$-equivalence class $C$, and each
  $T_m$-equivalence class $C'$ where $m>n$, if $C \cap C' \neq
  \varnothing$ then $C \subseteq C'$; and
\item[\rm (2)] For each $T_n$-equivalence class $C$ there is a
  $T_{n+1}$-equivalence class $C'$ such that $C \subseteq C'\setminus
  \partial C'$.
\end{enumerate}
\end{cor}

We mention that unlayered toast (defined in Section \ref{sec:toast})
does exist and thus $F(2^{\Z^2})$ does have Borel chromatic number $3$.
This result will appear in an upcoming paper.

It is our opinion that this is only the beginning of nontrivial
results about countable Borel equivalence relations that can be proved
using forcing. It is curious to note the tension that the first five
``positive'' results all state the existence of points with certain
regularity properties, whereas the last three ``negative'' results
state the nonexistence of regular structure on orbits. Of course, the
positive results are all obtained by generically building such
elements in the generic extension (and then asserting their existence
in the ground model by absoluteness), and it is known that the results
do not hold for comeager or conull sets of reals. Thus what we are
using is some method that goes beyond the usual measure and category
arguments.

One of the central notions of the theory of countable Borel
equivalence relations is that of Borel reducibility, which is entirely
missing in the narrative of this paper, but is in fact an important
motivation. All known methods to prove nonreducibility results for
countable Borel equivalence relations have been measure-theoretic (it
is well known that category arguments would not work). But
measure-theoretic arguments have their limitations. There have been
persistent attempts by researchers to invent new methods that are not
measure-theoretic. For instance, recent work of S. Thomas
\cite{Thomas} and A. Marks \cite{Marks2} 
explore the use of Martin's
ultrafilter and its generalizations as a largeness notion (see also \cite{Marks} for other
recent uses of determinacy in the study of Borel equivalence relations).
The forcing methods presented in
this paper can also be viewed as an attempt in this direction.

\section{Preliminaries}

In this section we present some preliminaries that will be used
throughout the rest of the paper. More background terminology and
results will be recalled as needed in subsequent sections.

\subsection{Countable Borel equivalence relations and group actions}

In this paper we will be concerned mainly with countable Borel
equivalence relations. Let $X$ be a Polish space and $E$ an
equivalence relation on $X$. $E$ is {\em Borel} if it is a Borel
subset of $X\times X$. $E$ is {\em countable} if each $E$-equivalence
class is countable. For $x\in X$, we let $[x]_E$ denote the
$E$-equivalence class of $x$, i.e.,
$$ [x]_E=\{ y\in X\,:\, xEy\}. $$ When there is no ambiguity we will
omit the subscript and only write $[x]$.

Countable Borel equivalence relations typically arise from orbit
equivalence relations of countable group actions. If $\Gamma$ is a
countable discrete group and $\Gamma\acts X$ is a Borel action of
$\Gamma$ on a Polish space $X$, then the {\em orbit equivalence
  relation} $E^X_\Gamma$ defined by
$$ E^X_\Gamma=\{(x, y)\in X\times X\,:\, \exists g\in \Gamma\, (g\cdot
x=y)\} $$ is obviously a countable Borel equivalence
relation. Conversely, by a well-known theorem of Feldman--Moore, every
countable Borel equivalence relation is of the form $E^X_\Gamma$ for
some Borel action $\Gamma\acts X$ of a countable group $\Gamma$. For
this reason, whenever we speak of a countable Borel equivalence
relation $E$ we assume that there has been fixed a Borel action of a
countable group $\Gamma\acts X$ so that $E=E^X_\Gamma$. For any $x\in
X$, note that $[x]=\Gamma\cdot x$; we also refer to $[x]$ as the {\em
  orbit} of $x$.

A particularly important case for this paper is the action of
$\Gamma=\Z^d$ on $X=2^{\Z^d}$.  In this case we write $E_{\Z^d}$ for
the orbit equivalence relation $E^X_\Gamma$. We will frequently
restrict the action to the free part $F(2^{\Z^d})$ (defined formally
below), and we will also write $E_{\Z^d}$ for the restriction of the
orbit equivalence relation to the free part (which is also a Polish
space). The precise meaning will be clear from the context.

If $\Gamma\acts X$ and $\Gamma\acts Y$ are two actions of $\Gamma$ on
Polish spaces $X$ and $Y$, respectively, a {\em $\Gamma$-map}, or an
{\em equivariant map}, from $X$ to $Y$ is a map $\varphi: X\to Y$ such
that for all $g\in \Gamma$ and $x\in X$,
$$ \varphi(g\cdot x)=g\cdot \varphi(x). $$ If in addition $\varphi$ is
injective, it will be called a {\em $\Gamma$-embedding} or an {\em
  equivariant embedding}.

For a countable group $\Gamma$, the {\em Bernoulli shift} of $\Gamma$
is the action $\Gamma\acts 2^\Gamma$ defined by
$$ (g\cdot x)(h)=x(g^{-1}h) $$ for $x\in 2^\Gamma$ and $g,
h\in\Gamma$. A closely related action $\Gamma\acts
2^{\Gamma\times\omega}$ is defined by
$$ (g\cdot x)(h, n)=x(g^{-1}h, n) $$ for $x\in
2^{\Gamma\times\omega}$, $g, h\in\Gamma$, and $n\in \omega$. A theorem
of Becker--Kechris states that this latter action is a {\em universal}
Borel $\Gamma$-action. That is, for any Borel action $\Gamma\acts X$
of $\Gamma$ on a Polish space $X$, there is a Borel $\Gamma$-embedding
from $X$ into $2^{\Gamma\times \omega}$. In view of this, any
$\Gamma$-action on a Polish space $X$ is Borel isomorphic to the
action of $\Gamma$ restricted to an invariant Borel subset of
$2^{\Gamma\times \omega}$.

\subsection{Aperiodicity, hyperaperiodicity, and minimality}

Let $\Gamma$ be a countable group, $X$ a Polish space, and
$\Gamma\acts X$ a Borel action.  An element $x\in X$ is {\em
  aperiodic} if for any nonidentity $g\in \Gamma$, $g\cdot x\neq
x$. The set of all aperiodic elements of $X$ is called the {\em free
  part} of $X$, and is denoted as $F(X)$. When $F(X)=X$ we say that
the action is {\it free}.

An element $x\in X$ is {\em hyperaperiodic} if the closure of its
orbit is contained in the free part of $X$, i.e.,
$\overline{[x]}\subseteq F(X)$. A hyperaperiodic element of
$2^\Gamma$ is sometimes also called a {\em $2$-coloring} on
$\Gamma$ for the following reason.

\begin{lemma}[\cite{GJS1}]
A point $x \in 2^\Gamma$ is hyperaperiodic if and only if
for all $e_\Gamma \neq s \in \Gamma$, there is a finite
$T\subseteq \Gamma$ such that
$$\forall g \in \Gamma \ \exists t \in T \quad x(g s t) \neq x(g t).$$
\end{lemma}

The above
combinatorial property emphasizes $x$ as a function assigning two
colors $0, 1$ to elements of $\Gamma$ in a way that for any shift $s$,
the pair of elements $g$ and $gs$ might not have different colors, but
it only takes a ``small" perturbation $t$ to take the pair to a new
pair $gt$ and $gst$ with different colors. When the underlying space
is $2^\Gamma$, we will use the two terms hyperaperiodic and 2-coloring
interchangeably.

Unfortunately, in this paper we will also consider graph colorings in
the usual sense that adjacent vertices have different colors. If $k$
many colors are used, we will refer to such colorings as {\em graph
  $k$-colorings} or {\em chromatic $k$-colorings}.

The action $\Gamma\acts X$ is {\it minimal} if every orbit is dense,
i.e., $\overline{[x]}=X$ for every $x\in X$. In general, we call an element
$x\in X$ {\it minimal} if the induced action of $\Gamma$ on
$\overline{[x]}$ is minimal. When $X$ is a compact Polish space and
the action is continuous, an application of Zorn's lemma shows that
there always exist minimal elements. In fact, when $X$ is compact Polish
(or even compact with a wellordered base) we can prove this in $\zf$ 
(i.e., we don't need any form of $\ac$ to prove this).

\begin{fact} [$\zf$]
Let $X$ be a compact $T_2$ topological space with a 
wellordered base $\{ U_\eta\}_{\eta <\lambda}$.
Let $\Gamma$ be a group and $(g,x)\mapsto g \cdot x \in X$ a continuous action 
of $\Gamma$ on $X$. Then there is an $x \in X$ which is minimal.
\end{fact}

If the action of $\Gamma$ is only Borel, then the same conclusion holds
if $X$ is compact and Polish \cite[Lemma 2.4.4]{GJS2}.

\begin{proof}
Let $K_0=X$. We define by transfinite recursion on the ordinals a sequence
of non-empty compact sets $K_\alpha\subseteq X$ which are decreasing, in that if $\alpha <\beta$
then $K_\beta \subseteq K_\alpha$, and also invariant, in that if $x \in K_\alpha$ then 
$[x]=\{ g\cdot x \colon g \in \Gamma\} \subseteq K_\alpha$. For $\alpha$ limit we set $K_\alpha =\bigcap_{\beta<\alpha}
K_\beta$. For the successor case, suppose $K_\alpha$ is defined, and is non-empty, compact, and invariant.
If $K_\alpha$ is minimal, we are done. Otherwise there is a least $\eta <\lambda$
such that $U_\eta \cap K_\alpha\neq \emptyset$ and $K_\alpha- \sat(U_\eta)\neq \emptyset$
(here $\sat(U)=\{ y \colon \exists x \in U \exists g \in \Gamma\ y=g\cdot x\}$ is the saturation of
$U$ under the equivalence relation on $X$ generated by $\Gamma$). This exists since we are assuming there is a
non-empty, compact (hence closed), invariant $K\subsetneq K_\alpha$. Let 
$K_{\alpha+1}=K_\alpha- \sat(U_\eta)$. Since the action of $\Gamma$ is continuous, $\sat(U_\eta)$
is open, so $K_{\alpha+1}$ is non-empty and compact. It is also
invariant (the difference of two invariant sets), and properly contained in $K_\alpha$. 
As the sets $K_\alpha$ are decreasing, the sequence must terminate in some $K_\alpha$
which is minimal.
\end{proof}

\begin{cor}[$\zf$]
A continuous action of a Polish group $\Gamma$ on a compact Polish space $X$
has a minimal element.
\end{cor}

In the case of $2^\Gamma$, minimality
is captured by the following combinatorial condition. We will use the following
fact repeatedly.

\begin{lemma}
A point $x \in 2^\Gamma$ is minimal if and only if for every finite
$A \subseteq \Gamma$ there is a finite $T \subseteq \Gamma$ such that
$$\forall g \in \Gamma \ \exists t \in T \ \forall a \in A \quad
x(g t a) = x(a).$$
\end{lemma}

\begin{proof}
This is well known and follows from a simple compactness argument.
A proof can be found, for example, in \cite{GJS2}.
\end{proof}

It was proved in \cite{GJS2} that minimal $2$-colorings exist on every
countable group $\Gamma$.

\subsection{Borel complete sections, Borel graphs, and geometry on orbits}
For a Polish space $X$ with a countable Borel equivalence relation
$E$, a {\em complete section} $S$ is a subset of $X$ that meets every
orbit of $X$, i.e., for any $x\in X$, $S\cap [x]\neq \varnothing$. A
useful fact in the theory of countable Borel equivalence relations is
the theorem of Slaman--Steel that for any countable Borel equivalence
relation $E$ with only infinite equivalence classes, there is a
decreasing sequence of Borel complete sections
$$ S_0\supseteq S_1\supseteq S_2\supseteq \cdots\cdots $$ such that
$\bigcap_n S_n=\varnothing$ \cite{SlamanSteel}.

A particularly interesting collection of examples is given by the
Bernoulli shifts of $\mathbb{Z}^d$ on $F(2^{{\mathbb Z}^d})$ for
$d\geq 1$. The Borel complete sections given by the Slaman-Steel
theorem lead to a quick proof that the orbit equivalence relation 
in the case $d=1$ is
{\em hyperfinite}. Recall that a countable equivalence relation
on a standard Borel space $X$ 
is hyperfinite if 
there is a an increasing sequence of Borel equivalence
relations
$$ R_0\subseteq R_1\subseteq R_2\subseteq \cdots\cdots $$ on $X$ with all
$R_n$-equivalence classes finite, such that $E=\bigcup_n R_n$.
Weiss later showed that this holds for all $d$ (from which
it follows that any Borel action of $\Z^d$ is hyperfinite). 
In these examples we rely on the geometric structure of the
Cayley graph of the group to understand the behavior of the orbit
equivalence relation. Complete sections are frequently built to posses
properties of geometric significance and for this reason are
informally called sets of {\it markers}.

The following notions and terminology are tools to study the geometric
structures. Fix $d\geq 1$. For an element $g=(g_1,\dots,
g_d)\in{\mathbb Z}^d$, let
$$ \|g\|=\sum_{i=1}^d |g_i|. $$ The metric induced by this norm is
often called the {\it taxi-cab metric}.  If $x, y\in F(2^{{\mathbb
    Z}^d})$ are in the same orbit, then there is a unique
$g_{x,y}\in{\mathbb Z}^d$ with $g_{x,y}\cdot x=y$, and we define
$\rho(x,y)=\|g_{x,y}\|$. If $y\not\in[x]$, we just let
$\rho(x,y)=+\infty$. This $\rho$ is thus a distance function that is a
metric on each orbit.

For $A\subseteq F(2^{{\mathbb Z}^d})$ we also define
$\rho(x,A)=\min\{\rho(x,y)\,:\, y\in A\}$. If $A$ is a complete
section, then $\rho(x,A)<+\infty$ for any $x$. If $\{S_n\}$ is a
decreasing sequence of Borel complete sections with $\bigcap_n
S_n=\varnothing$ as in the Slaman--Steel theorem, then for any $x\in
F(2^{{\mathbb Z}^d})$, we have $\lim_n \rho(x, S_n)=+\infty$. In fact,
the function $\varphi_x(n)=\rho(x, S_n)$ is a monotone increasing
function diverging to infinity for each $x$. Our first theorem in this
paper, which is presented in the next section, implies that this
function grows arbitrarily slowly. More precisely, given any function
$f(n)$ with $\limsup_n f(n)=+\infty$, we show that there is $x$ with
$\varphi_x(n)<f(n)$ for infinitely many values of $n$.

In general, for any finitely generated group $\Gamma$ with a finite
symmetric generating set $S$ (meaning $\gamma^{-1}\in S$ for every
$\gamma\in S$), a number of standard objects can be associated with
the space $2^\Gamma$. First, there is the {\it Cayley graph}
$C_S(\Gamma)=(\Gamma, D)$ with $\Gamma$ as the vertex set and with the
edge relation $D$ defined by
$$ (g,h)\in D\iff \exists \gamma\in S\ ( g=\gamma h). $$ The Cayley
graph induces a Borel graph $C_S(2^\Gamma)=(2^\Gamma, \tilde{D})$ on
$2^\Gamma$, where the edge relation $\tilde{D}$ is defined as
$$ (x,y)\in \tilde{D}\iff \exists \gamma\in S\ (x=\gamma\cdot y). $$
The geodesics in $C_S(2^\Gamma)$ give a distance function
$\rho_\Gamma$, i.e., $\rho_\Gamma(x,y)$ is the length of the shortest
path from $x$ to $y$ in $C_S(2^\Gamma)$. Note that the distance
function $\rho$ defined above for $2^{\Z^d}$ is an example of the more
general $\rho_\Gamma$, with $S$ being the set of standard generators
for $2^{\Z^d}$.

If $A\subseteq 2^\Gamma$, the {\it boundary} of $A$, denoted $\partial
A$, is the set
$$ \partial A=\{ x\in A\,:\, \exists \gamma\in S\ \gamma\cdot x\not\in
A\}. $$

\section{Orbit Forcing}
We start with a very general forcing construction in which one
generically builds an element in a Polish space with a countable Borel
equivalence relation.

\begin{defn}
Let $E$ be a countable Borel equivalence relation on a Polish space
$X$, and let $x \in X$.  The {\em orbit forcing} $\bbP_x=\bbP_x^E$ is
defined by
$$ \bbP_x=\{U \subseteq X\,:\, \mbox{$U$ is open and } U \cap [x]_E
\neq \varnothing\} $$ with its elements ordered by inclusion, that is,
$U \leq U'$ iff $U \subseteq U'$.
\end{defn}

Since $U \cap \ocl{[x]}\neq \varnothing$ iff $U \cap [x] \neq
\varnothing$, we can view the sets $U\cap \ocl{[x]}$ as the objects in
the forcing notion, in which case the forcing notion $\bbP_x$ can
simply be viewed as ordinary Cohen forcing on the closed subspace
$Y=\ocl{[x]}$ of $X$.  Thus, as with usual Cohen forcing, we can
regard forcing arguments using $\bbP_x$ as category arguments on the
space $\ocl{[x]}$. Nevertheless, we will see that the forcing proofs
can be more intuitive than category arguments.

As remarked in the previous section, we can view the countable Borel
equivalence relation $E$ as coming from a Borel action of a countable
group $\Gamma$, and view the space $X$ as a certain invariant Borel
subset of $2^{\Gamma\times\omega}$. Now if $G$ is $\bbP_x$-generic
over $V$, the space $X^{V[G]}$ continues to be a standard Borel space
and $E^{V[G]}$ continues to be a countable Borel equivalence
relation. Moreover, the generic $G$ can be identified with an element
$x_G\in X^{V[G]}$. With a slight abuse of terminology we will refer to
$x_G$ as a generic element of $X$ for the orbit forcing $\bbP_x$. Note
that we always have that $x_G \in \ocl{[x]}_E$, where both the orbit
and the closure are computed in $V[G]$.

As a first application of the orbit-forcing we present a general
result about sequences of complete sections in equivalence relations
generated by continuous actions of countable groups on compact spaces.
This includes the case of Bernoulli shift actions on $2^{\Gamma}$ and
$F(2^\Gamma)$ since there exist compact invariant sets $X \subseteq
F(2^\Gamma)$ \cite{GJS1,GJS2}.

\begin{thm} \label{dtm}
Let $\Gamma$ be a countable group, $X$ a compact Polish space, and
$\Gamma \acts X$ a continuous action giving rise to the orbit
equivalence relation $E$.  Let $(A_n)_{n \in \N}$ be a sequence of
finite subsets of $\Gamma$ such that every finite subset of $\Gamma$
is contained in some $A_n$. Let $(S_n)_{n \in \N}$ be a sequence of
Borel complete sections of $E$. Then there is an $x \in X$ such that
for infinitely many $n$ we have $A_n\cdot x \cap S_n \neq \varnothing$.
\end{thm}

\begin{proof}
Since $X$ is compact, we may fix an $x \in X$ which is minimal. Let
$\bbP=\bbP_x$ be the the corresponding orbit forcing. Let $\kappa$ be
a large enough regular cardinal and $Z\prec V_\kappa$ a countable
elementary substructure containing $\Gamma$ and all the real
parameters used in the definitions of $X$ (as a Borel subspace of
$2^{\Gamma\times\omega}$), the action $\Gamma\acts X$, and the
sequences $\{ A_n\}$ and $\{S_n\}$.  Let $\pi \colon M \to Z$ be the
inverse of the transitive collapse. Let $\pi(\bbP')=\bbP$. Fix an
arbitrary $G$ which is $\bbP'$-generic over $M$, and let $x_G$ be the
unique element of $\bigcap G$. Let $\dot{x}_G$ be the canonical
$\bbP'$-name for $x_G$. Note that for any $\gamma\in \Gamma$,
$\gamma\cdot G$ is also $\bbP'$-generic over $M$, and we have
$M[G]=M[\gamma\cdot G]$ and $x_{\gamma\cdot G}=\gamma\cdot x_G$.

Working in $V$, we will define a sequence $U_n$ of conditions in
$\bbP'$ such that the filter $H$ generated by $\{U_n\}$ is
$\bbP'$-generic over $M$, and an increasing sequence $i_n$ of integers
such that $$U_n \forces (A_{i_n} \cdot \xd_G \cap S_{i_n} \neq
\varnothing).$$
To see this suffices, let $x_H$ be the unique element of $\bigcap_n
U_n$. By assumption, $x_H$ is $\bbP'$-generic over $M$. From the
forcing theorem, we have in $M[H]$ that $A_{i_n}\cdot x_H \cap
S_{i_n}\neq \varnothing$. Therefore, the statement $ \exists y \forall
n \, (A_{i_n}\cdot y\cap S_{i_n}\neq\varnothing) $ is true in
$M[H]$. Since this is a $\bs^1_1$ statement, by absoluteness it is
also true in $V$, which gives the desired result.

Let $D_0, D_1,\dots$ enumerate the dense subsets of $\bbP'$ which lie
in $M$.  To begin, let $i_0$ be such that the identity $e_{\Gamma}\in
A_{i_0}$. The statement that $S_{i_0}$ is a Borel complete section is
$\bp^1_1$, and hence by absoluteness it continues to be a Borel
complete section in $M[G]$. Fix any $y \in [x_G] \cap S_{i_0}$. There
is a certain $\gamma\in\Gamma$ such that $y=\gamma\cdot x_G$. Since
$y$ is also $\bbP'$-generic over $M$, there is a $U'_0 \in \bbP'$ such
that $U'_0 \forces (\xd_G \in S_{i_0})$. This gives $U'_0 \forces
(A_{i_0} \cdot \xd_G \cap S_{i_0}\neq \varnothing)$. Let $U_0
\subseteq U'_0$ be in $\bbP' \cap D_0$.

In general, suppose $U_n \in \bbP'$ and $i_n$ are given so that $U_n
\forces (A_{i_n} \cdot \xd_G \cap S_{i_n} \neq \varnothing)$. By
minimality of $x$, there is a finite $T \subseteq \Gamma$ such that
for all $z \in \overline{[x]}$ there is a $t \in T$ with $t \cdot z
\in U_{n}$. Fix such a $T$. The statement $\forall
z\in\overline{[x]}\ (T\cdot z\cap U_n\neq\varnothing)$ is $\bp^1_1$
and continues to be true in $M[G]$. Since $x_G\in \overline{[x]}$ in
$M[G]$, we have that for any $z\in[x_G]$, $T\cdot z\cap
U_n\neq\varnothing$. Let $i_{n+1}$ be such that $T^{-1} \subseteq
A_{i_{n+1}}$. Without loss of generality we may assume that
$i_{n+1}>i_n$.  Now fix any $y \in [x_G] \cap S_{i_{n+1}}$. Fix $t \in
T$ such that $t \cdot y \in U_{n}$.  So, $t^{-1}\cdot (t\cdot y) \in
S_{i_{n+1}}$. As $t\cdot y$ is generic and $t^{-1} \in A_{i_{n+1}}$ we
have that there is a $U'_{n+1} \subseteq U_{n}$ in $\bbP'$ such that
$U'_{n+1} \forces (A_{i_{n+1}} \cdot \xd_G \in S_{i_{n+1}})$. Let
$U_{n+1} \subseteq U'_{n+1}$ be in $\bbP'\cap D_{n+1}$. This completes
the construction of the $U_n$ and finishes the proof of the theorem.
\end{proof}

We have the following immediate corollary concerning complete sections
in $F(2^{\Z^n})$.

\begin{cor} \label{cor:dtm}
Let $f \colon \N \to \N$ be such that $\lim \sup_n f(n)=+\infty$. Let
$\{ S_n\}$ be a sequence of Borel complete sections of $F(2^{\Z^d})$.
Then there is an $x \in F(2^{\Z^d})$ such that for infinitely many $n$
we have $\rho(x,S_n)<f(n)$.
\end{cor}

\begin{proof}
Let $x$ be any 2-coloring (or hyperaperiodic element) in
$2^{\Z^d}$. Then $X=\overline{[x]}$ is a compact invariant
subspace. Apply Theorem~\ref{dtm} to $X$ with $A_n=\{
\gamma\in\Z^d\,:\, \|\gamma\|<f(n)\}$.
\end{proof}

\begin{rem} \label{rem:dtm} (1) The proof of Theorem~\ref{dtm} 
still works if each $S_n$ is just 
assumed to be absolutely $\bd^1_2$, instead of Borel.
By this we mean there are $\bs^1_2$ and $\bp^1_2$ statements 
$\varphi$ and $\psi$ respectively which define $S_n$,
and such that $\varphi$, $\psi$ continue to define complimentary sets
in all forcing extension $V[G]$ of $V$.  
There are two applications of absoluteness regarding the $S_n$
in the proof of Theorem~\ref{dtm}. 
In the first application (getting the $S_n$ to be complete sections in $M[G]$), 
the $S_n$ needs to be $\bp^1_2$, 
and in the second application (lifting the property of $y$ from 
$M[H]$ to $V$), it needs to be $\bs^1_2$.

(2) The proof of Theorem~\ref{dtm} shows that we may weaken the
hypothesis of Corollary~\ref{cor:dtm} that the $S_n$ are complete
sections to the statement that for each $x \in F(2^{\Z^d})$ and for
each $n$ there is an $m \geq n$ such that $S_m \cap [x] \neq
\varnothing$. However, we need to assume now that $\liminf_n
f(n)=+\infty$.
\end{rem}

As we mentioned above, a forcing argument using $\bbP_x$ is
essentially the same as a category argument on the subspace
$\ocl{[x]}$ of the Polish space $X$. In particular, the above proof
can be given in purely topological terms.  We feel it is worth
presenting this alternative argument explicitly.  We first recall a
concept and record a simple lemma.

\begin{defn}
Consider a Borel action of a countable group $\Gamma$ on a Polish
space $X$. A point $x \in X$ is \emph{recurrent} if for every open set
$U \subseteq X$ with $U\cap [x]\neq\varnothing$, there is a finite $T
\subseteq \Gamma$ so that for all $y\in [x]$ there is $t \in T$ with
$t \cdot y \in U$.
\end{defn}

\begin{lemma}\label{recurrent}
Let $\Gamma$ be a countable group, $X$ a compact Polish space, and
$\Gamma\acts X$ a continuous, minimal action. Then every $x \in X$ is
recurrent.
\end{lemma}

\begin{proof}
Fix a point $x \in X$ and a non-empty open set $U \subseteq
X$. Enumerate $\Gamma$ as $\gamma_1, \gamma_2, \ldots$ and set $T_n =
\{\gamma_i \,:\, 1 \leq i \leq n\}$. Towards a contradiction, suppose
that for every $n$ there is $x_n \in [x]$ with $T_n \cdot x_n \cap U =
\varnothing$. Since $X$ is compact, there is an accumulation point $y$
of the sequence $x_n$. Now for any $i \in \N$ we have $\gamma_i \cdot
x_n \not\in U$ for every $n \geq i$. Since $\Gamma$ acts continuously and
$U$ is open, it follows that $\gamma_i \cdot y \not\in U$. Thus the
orbit of $y$ does not meet $U$ and hence is not dense, a contradiction
to the minimality of the action.
\end{proof}

To give a topological proof of Theorem~\ref{dtm}, we will make use of
the strong Choquet game. Let us recall this game. The Strong Choquet
game on a topological space $X$ consists of two players (I and II) who
play in alternating turns. On each of player I's turns, player I plays
a pair $(U, x)$ consisting of an open set $U$ and a point $x \in
U$. Player II plays an open set on each of her turns. A play of the
game is illustrated below.
$$ \begin{array}{lcccccccl} \mbox{I} & (U_0, x_0) & & (U_1, x_1) & &
  \cdots & & \cdots & \\ \mbox{II} & & V_0 & & V_1 & & \cdots &
  &\cdots
\end{array}
$$ The game requires that $U_{n+1}\subseteq V_n$ and $x_n\in
V_n\subseteq U_n$ for all $n$. The first player breaking these rules
loses. Player II wins the game if and only if $\bigcap_n U_n=\bigcap_n
V_n\neq\varnothing$.  A space is called {\em strong Choquet} if player
II has a winning strategy for this game.  It is an easy fact that
every completely metrizable space is strong Choquet (cf., e.g.,
\cite[\S 4.1]{Gao}).

We now present the alternative argument for Theorem~\ref{dtm}.  By
considering $\ocl{[x]}$ where $x \in X$ is minimal, we may assume that
$\Gamma$ acts continuously and minimally on $X$.  As Borel sets in a
Polish space have the Baire property, we can find a $\Gamma$-invariant
dense $G_\delta$ set $X' \subseteq X$ such that each set $S_n \cap X'$
is relatively open in $X'$. By Lemma~\ref{recurrent}, every $x \in X$
is recurrent and it is easy to see that this implies that every $x \in
X'$ is recurrent with respect to the relative topology of $X'$. We
will consider the strong Choquet game on $X'$. Since $X'$ is
$G_\delta$ in $X$, $X'$ is completely metrizable, and thus player II
has a winning strategy. Fix a winning strategy $\tau$ for player
II. From this point forward we will only work with $X'$, not $X$.  Let
$i_0$ be such that the identity $e_{\Gamma} \in A_{i_0}$.  To begin
the game, fix any $x_0 \in S_{i_0}\cap X' \subseteq A_{i_0}^{-1} \cdot
S_{i_0} \cap X'$ and have player I play the open set $U_0 =
A_{i_0}^{-1} \cdot S_{i_0} \cap X'$ and the point $x_0$. Such an $x_0$
exists since $X'$ is $\Gamma$-invariant and $S_{i_0}$ is a complete
section. Let $V_0 \subseteq X'$ be the open set played by player II
according to $\tau$. As $x_0$ is recurrent, there is a finite $T_0
\subseteq \Gamma$ so that for all $y \in [x_0]$ we have $T_0 \cdot y
\cap V_0 \neq \varnothing$. The largeness assumption on the $A_n$'s
implies that we can find $i_1 >i_0$ with $T_0^{-1} \subseteq A_{i_1}$.

The game now proceeds inductively. Assume that player I has played the
pairs $(U_0, x_0), (U_1, x_1), \ldots, (U_{n}, x_{n})$, player II has
played the open sets $V_0, V_1, \ldots, V_{n}$ according to her
winning strategy $\tau$, and that $i_{n+1}$ has been defined and the
following two conditions are satisfied:
\begin{enumerate}
\item[\rm (i)] for every $y \in [x_{n}]$ we have $A_{i_{n+1}}^{-1}
  \cdot y \cap V_{n} \neq \varnothing$;
\item[\rm (ii)] for all $y \in V_{n}$ and all $1 \leq m \leq n$ we
  have $A_{i_m} \cdot y \cap S_{i_m} \neq \varnothing$.
\end{enumerate}
Fix a point $x'_{n+1} \in S_{i_{n+1}} \cap [x_{n}]$. Such $x'_{n+1}$
exists since $S_{i_{n+1}}$ is a complete section. By clause (i) we may
pick a point $x_{n+1} \in A_{i_{n+1}}^{-1} \cdot x'_{n+1} \cap
V_{n}$. Have player I play the point $x_{n+1}$ and the open set
$$U_{n+1} = V_{n} \cap A_{i_{n+1}}^{-1} \cdot S_{i_{n+1}}.$$ Note that
for every $y \in U_{n+1}$ and every $1 \leq m \leq {n+1}$ we have
$A_{i_m} \cdot y \cap S_{i_m} \neq \varnothing$. Let $V_{n+1}$ be the
open set played by player II according to $\tau$. Then $x_{n+1} \in
V_{n+1}$ and by recurrence there is a finite $T_{n+1} \subseteq
\Gamma$ so that for all $y \in [x_{n+1}]$ we have $T_{n+1} \cdot y
\cap V_{n+1} \neq \varnothing$. Now we can find a number
$i_{n+2}>i_{n+1}$ with $T_{n+1}^{-1} \subseteq A_{i_{n+2}}$. This
completes the induction.

Since player II has followed $\tau$, we have that there is $x \in
\bigcap_n V_n$. Thus it follows from clause (ii) that for every $m \in
\N$ we have $A_{i_m} \cdot x \cap S_{i_m} \neq \varnothing$.  This
completes the alternative proof of Theorem~\ref{dtm}. \hfill $\square$

As a consequence of the methods of Theorem~\ref{dtm} we get the following result on
the existence of recurrent points in the range of factor maps.

\begin{thm}\label{recurrentfactor}
Let $\Gamma$ be a countable group, $X$ a compact Polish space,
$\Gamma\acts X$ a continuous action , $Y$ a Polish space, and
$\Gamma\acts Y$ a Borel action. Let $\varphi \colon X \to Y$ be a
Borel equivariant map.  Then there is an $x \in X$ such that
$\varphi(x)$ is a recurrent point of $Y$.
\end{thm}

\begin{proof}
 
Let $x\in X$ be minimal, and let $\bbP=\bbP_x$ and $\bbP'$ be as in
the proof of Theorem~\ref{dtm} (along with $M$, $\pi$, etc., with
codes for $X, Y$ and $\varphi$ in $M$).  Let $G$ be $\bbP'_x$-generic
over $M$.  Note that $x_G \in \ocl{[x]}$, and so $x_G$ is also a
minimal point. We claim that $\varphi(x_G)$ is recurrent.  Let
$V_0,V_1,\dots$ enumerate the basic open sets of $Y$ which intersect
$[\varphi(x_G)]$.  Let $S_n=\varphi^{-1}(V_n)$, so $S_n$ is a Borel
subset of $X$.  For each $n$, let $\gamma_n \in \Gamma$ and $U_n \in
\bbP'_x$ be such that $U_n \forces ( \gamma_n \cdot \xd_G \in S_n)$.
Since $x_G$ is in the open set $U_n$, and since $x_G$ is minimal, it
follows that there is a finite $T_n \subseteq \Gamma$ such that for
any $y =\gamma\cdot x_G \in [x_G]$ there is a $t \in T_n$ such that $t
\cdot y \in U_n$. Since $t \cdot y$ is also generic, if follows that
$\gamma_n \cdot (t \cdot y) \in S_n$. By equivariance of $\varphi$ we
have $(\gamma_nt\gamma)\cdot \varphi(x_G) \in V_n$. So, for all
$\gamma \in \Gamma$ there exists $h \in \gamma_nT_n$ such that $h\cdot
(\gamma\cdot \varphi(x_G)) \in V_n$.  This shows that $\varphi(x_G)$
is recurrent.
\end{proof}

\begin{cor}
For any countable group $\Gamma$, any Borel action $\Gamma\acts Y$ of
$\Gamma$ on a Polish space $Y$, and any Borel equivariant map
$\varphi\colon F(2^{\Gamma}) \to Y$, there is an $x \in F(2^{\Gamma})$
such that $\varphi(x)$ is recurrent.
\end{cor}

\begin{proof}
By \cite{GJS1, GJS2}, there is an invariant compact set $X \subseteq
F(2^{\Gamma})$. Now apply Theorem~\ref{recurrentfactor} to $X$.
\end{proof}

\begin{cor}\label{taurecurrent}
Let $\tau$ be a Polish topology on $F(2^\Gamma)$ having the same Borel
sets as the standard topology. Then there is a $\tau$-recurrent point.
\end{cor}

\begin{cor}\label{meetorbitrecurrent}
If $B \subseteq F(2^\Gamma)$ is a Borel complete section then $B$
meets some orbit recurrently, i.e., there is $x\in F(2^\Gamma)$ and
finite $T\subseteq \Gamma$ such that for any $y\in[x]$, $T\cdot y\cap
B\neq\varnothing$.
\end{cor}

\begin{proof}
Let $\tau$ be a Polish topology on $F(2^\Gamma)$ with $B\in \tau$ and
having the same Borel sets as the standard topology (cf., e.g.,
\cite[\S 4.2]{Gao}). Apply Corollary~\ref{taurecurrent}.
\end{proof}

Corollary~\ref{meetorbitrecurrent} rules out the existence of Borel
complete sections with certain geometric properties. The following is
an example.

\begin{cor} \label{cor:nn}
There does not exist $B \subseteq F(2^{\Z^2})$ with the following
properties:
\begin{enumerate}
\item[\rm (i)] Both $B$ and $F(2^{\Z^2})\setminus B$ are Borel
  complete sections, and
\item[\rm (ii)] For any $x\in B$ and $(\gamma, \eta)\in\N^2$,
  $(\gamma,\eta) \cdot x \in B$.
\end{enumerate}
\end{cor}

\begin{proof}
Assume toward a contradiction that $B$ satisfies both (i) and (ii). By
Corollary~\ref{meetorbitrecurrent} there is $x\in F(2^{\Z^2})$ and
finite $T\subseteq \Z^2$ such that for any $y\in[x]$, $T\cdot y\cap
B\neq\varnothing$. Let $z\in [x]\setminus B$. Such a $z$ exists since
$F(2^{\Z^2})\setminus B$ is also a complete section. By (ii), for any
$(\gamma, \eta)\in \N^2$, $(-\gamma, -\eta)\cdot z\not\in B$. Now let
$n>\|t\|$ for all $t\in T$, and consider $y=(-n, -n)\cdot z$. The
property of $z$ implies that $T\cdot y\cap B=\varnothing$, a
contradiction.
\end{proof}

We note that it is also possible to give a non-forcing proof
of  Corollary~\ref{cor:nn}. Namely, if such a $B$ existed,
then the subequivalence relation of $F(2^{\Z^2})$
generated by the group element $g=(1,1)$ would be smooth 
(that is, have a Borel selector). However, a simple category argument shows that 
there cannot be such a Borel selector for this relation.

In all forcing arguments in the rest of this paper we will skip the
metamathematical details we presented in the proofs of
Theorems~\ref{dtm} and \ref{recurrentfactor}. Specifically, instead of
using the countable elementary substructure $M$ and the forcing
version $\bbP'$ in $M$ (in the case of Cohen forcing we in fact have
$\bbP'=\bbP$), we will just use $\bbP$ and pretend that we can find a
generic $x$ for $\bbP$ over $V$. In reality, we should take $x$ to be
$M$-generic for $\bbP'$ and then use absoluteness between $M[x]$ and
$V$ as we did in Theorem~\ref{dtm}.  Since these details do not vary
in the arguments, we shall henceforth omit them.

\section{Borel Layered Toast} \label{sec:toast}

In this short section we present another application of
Corollary~\ref{cor:dtm} on the non-existence of certain types of {\it
  strong marker structure} on $F(2^{\Z^d})$. The name ``toast'' for
the type of structure defined below was coined by B.\ Miller. We will define two
versions of this notion, the general or ``unlayered'' toast structure,
and the more restrictive notion of ``layered'' toast.  These are both
strong types of marker structures to impose on the orbits of
$F(2^{\Z^d})$. We can consider both the Borel as well as the clopen
versions of these notions, which leads to four separate questions
concerning the existence of these structures. It turns out that a
Borel unlayered toast structure does exist, but the answers are no
for all the other existence questions.  We present the proof for the
nonexistence of Borel layered toast here; the other results will be
presented in a forthcoming paper.

We note that the notion of toast arose naturally through its connections
with interesting problems in Borel combinatorics.
For example, the existence of Borel unlayered toast, which
will be proved in an upcoming paper, gives a proof that there is a Borel
chromatic $3$-coloring of $F(2^{\Z^d})$ (and so $F(2^{\Z^d})$ has Borel
chromatic number $3$). These toast structures have been constructed modulo
meager sets and modulo $\mu$-null sets by C.\ Conley and B.\ Miller and used
to bound the Baire-measurable and $\mu$-measurable chromatic numbers
of many Borel graphs \cite{CMi}.

First we make precise the notion of a toast marker structure.

\begin{defn} \label{def:toast}
Let $\{T_n\}$ be a sequence of subequivalence relations of $E_{\Z^d}$
on some subsets $\dom(T_n)\subseteq F(2^{\Z^d})$ with each
$T_n$-equivalence class finite. Assume $\bigcup_n \dom(T_n)=
F(2^{\Z^d})$. We say $\{ T_n\}$ is a {\em (unlayered) toast} if:
\begin{enumerate}
\item[(1)] \label{toast1} For each $T_n$-equivalence class $C$, and
  each $T_m$-equivalence class $C'$ where $m>n$, if $C \cap C' \neq
  \varnothing$ then $C \subseteq C'$.

\item[(2)] \label{toast2} For each $T_n$-equivalence class $C$ there
  is $m>n$ and a $T_m$-equivalence class $C'$ such that $C \subseteq
  C'\setminus \partial C'$.
\end{enumerate}

We say $\{ T_n\}$ is a {\em layered toast} if, instead of (2) above,
we have
\begin{enumerate}
\item[(2')] For each $T_n$-equivalence class $C$ there is a
  $T_{n+1}$-equivalence class $C'$ such that $C \subseteq C'\setminus
  \partial C'$.
\end{enumerate}

\end{defn}

Figure~\ref{fig:toast} illustrates the definitions of layered and 
unlayered toast.

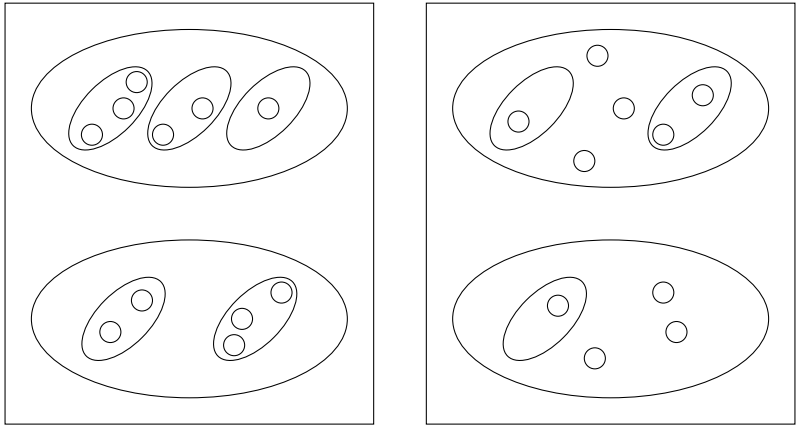
\begin{figure}[ht]
\centering
\begin{tikzpicture}[scale=0.035]

\draw (-150,-80) rectangle (-10,80);
\draw (10,-80) rectangle (150,80);

\draw (-80, 40) ellipse [x radius=60, y radius=30];
\draw (-80, -40) ellipse [x radius=60, y radius=30];
\draw[rotate around={45:(-110,40)}]   (-110,40) ellipse [x radius=20, y radius=10];
\draw[rotate around={45:(-80,40)}]   (-80,40) ellipse [x radius=20, y radius=10];
\draw[rotate around={45:(-50,40)}]   (-50,40) ellipse [x radius=20, y radius=10];
\draw (-100,50) circle (4);
\draw (-105,40) circle (4);
\draw (-117,30) circle (4);
\draw (-90,30) circle (4);
\draw (-75,40) circle (4);
\draw (-50,40) circle (4);
\draw[rotate around={45:(-105,-40)}]   (-105,-40) ellipse [x radius=20, y radius=10];
\draw[rotate around={45:(-55,-40)}]   (-55,-40) ellipse [x radius=20, y radius=10];
\draw (-110,-45) circle (4);
\draw (-98,-33) circle (4);
\draw (-63,-50) circle (4);
\draw (-60,-40) circle (4);
\draw (-45,-30) circle (4);

\draw (80, 40) ellipse [x radius=60, y radius=30];
\draw (80, -40) ellipse [x radius=60, y radius=30];
\draw[rotate around={45:(110,40)}]   (110,40) ellipse [x radius=20, y radius=10];
\draw[rotate around={45:(50,40)}]   (50,40) ellipse [x radius=20, y radius=10];
\draw (100,30) circle (4);
\draw (115,45) circle (4);
\draw (45,35) circle (4);
\draw (75,60) circle (4);
\draw (85,40) circle (4);
\draw (70,20) circle (4);
\draw[rotate around={45:(55,-40)}]   (55,-40) ellipse [x radius=20, y radius=10];
\draw (60,-35) circle (4);
\draw (74,-55) circle (4);
\draw (105,-45) circle (4);
\draw (100,-30) circle (4);

\end{tikzpicture}
\caption{(a) layered toast \hspace{50pt} (b) general toast} \label{fig:toast}
\end{figure}

\begin{thm}
There is no Borel layered toast on $F(2^{\Z^d})$.
\end{thm}

\begin{proof}
Suppose $\{T_n\}$ was a sequence of Borel subequivalence relations of
$E_{\Z^d}$ on some subsets $\dom(T_n)\subseteq F(2^{\Z^d})$ forming a
layered toast structure on $F(2^{\Z^d})$. For each $n$ let $\partial
T_n$ be the the union of all boundaries of the $T_n$-equivalence
classes. For $x \in F(2^{\Z^d})$, let $f_x \colon \N \to \N$ be
defined by $f_x(n)=\rho (x, \partial T_n)$ if $x \in \dom(T_n)$ and
$f_x(n)=0$ otherwise.  This is well-defined as each $T_n$-equivalence
class is finite. Since $\{T_n\}$ is a layered toast, by (2') we have
$\dom(T_n)\subseteq \dom(T_{n+1})$ for all $n$.  For $x \in
F(2^{\Z^d})$, let $n_0$ be large enough so that $x \in
\dom(T_{n_0})$. We claim that for $n \geq n_0$ that
$f_x(n)<f_x(n+1)$. To see this, let $n \geq n_0$, and let $a=f_x(n)$.
Let $g \in \Z^n$ with $\| g \| \leq a$.  It follows easily from the
definitions of $a$ and $\partial T_n$ that $g \cdot x$ is
$T_n$-equivalent to $x$ (if we choose a path $p$ from $\vec{0}$ to $g$
of length $a$, then by an easy induction along the path we have that
$g' \cdot x$ is $T_n$-equivalent to $x$ for all $g'$ in $p$).  So,
from property~(2') we have that $g \cdot x \notin \partial
T_{n+1}$. Thus, $\rho(x,\partial T_{n+1})>a$.  So, for all $x \in
F(2^{\Z^d})$ and all sufficiently large $n$ (which may depend on $x$)
we have $f_x(n)<f_x(n+1)$.

If we let $f\colon \N \to \N$ be the function $f(n)=\sqrt{n}$, then
for all $x \in F(2^{\Z^n})$ we have that for all but finitely many $n
\in \N$ that $\rho(x, \partial T_n) > f(n)$. This violates
Corollary~\ref{cor:dtm}.
\end{proof}

\section{Cohen Forcing and Bounded Geometry of Marker Regions}

In this section we use forcing to prove a nonexistence theorem for
marker regions in $F(2^{\Z^2})$ that are of regular shape. A version
of this theorem was stated without proof as Theorem~3.5 of
\cite{GaoJackson}. We will have to recall a good amount of terminology
and results from \cite{GaoJackson}. But the forcing used is going to
be just Cohen forcing on a countable group $\Gamma$.

Given a countable group $\Gamma$ and $k\leq \omega$, the Cohen forcing
$\bbP_\Gamma(k)$ is defined by
$$ \bbP_\Gamma(k)=\{ p\in k^{\dom(p)}\,:\,
\mbox{$\dom(p)\subseteq\Gamma$ is finite}\} $$ with the order of
inverse inclusion, that is, $p\leq q$ iff $p\supseteq q$.

If $G$ is $\bbP_\Gamma(k)$-generic over $V$, and $x_G=\bigcup G$, then
$x_G\in F(k^\Gamma)$. This is because, for any $\gamma\in\Gamma$, the
set
$$D_\gamma=\{ p\in \bbP_\Gamma(k)\,:\, \exists g\in
\dom(p)\cap\gamma^{-1}\cdot\dom(p)\ [\,p(g)\neq p(\gamma\cdot
  g)\,]\} $$ is dense in $\bbP_\Gamma(k)$. Thus the generic real is
always an aperiodic element of $k^\Gamma$.

We recall some facts about the orthogonal marker construction of
\cite{GaoJackson}. The construction was done on $F(2^{\Z^d})$ for any
$d\geq 1$. Here, we focus on $d=2$ for simplicity. We continue to use
$E_{\Z^2}$ to denote the orbit equivalence relation on $F(2^{\Z^2})$
given by the Bernoulli shift action of $\Z^2$. By a {\it finite
  subequivalence relation} on $F(2^{\Z^2})$ we mean an equivalence
relation $\mathcal{R}\subseteq E_{\Z^2}$ with all equivalence classes
finite.  If $R$ is an equivalence class of $\mathcal{R}$ and $x\in R$,
then we consider the finite subset of $\Z^2$ defined by
$$ S_x=\{(\gamma,\eta)\in\Z^2\,\:\, (\gamma,\eta)\cdot x\in R\}. $$ We
can speak of the {\it shape} of $S_x$, e.g., $S_x$ is a {\it
  rectangle} if it is of the form $[a,b]\times [c,d]\subseteq
\Z^2$. It is obvious that the shape of $S_x$ does not depend on the
choice of $x$, since these sets are translates of each other with
different choices of reference points. Thus we often abuse the
terminology and just speak of the {\it shape} of an equivalence class
$R$.

In the orthogonal marker construction one produces a sequence $\sR_n$
of relatively clopen finite subequivalence relations on
$F(2^{\Z^2})$. Here the term {\it relatively clopen} means that for
every $g \in \Z^2$, the set of $x \in F(2^{\Z^2})$ with $x \sR_n g \cdot x$
is clopen. The $\sR_n$ equivalence classes are also called {\it
  marker regions}. There is a {\it scale} $d_n \in \N$ associated to
each $\sR_n$, in the sense that each equivalence class $R$ of $\sR_n$
restricted to some orbit $[x]$ is {\it roughly a rectangle on the
  scale} $d_n$, that is, there is a rectangle $R'$ such that the
$\rho$-Hausdorff distance between $R$ and $R'$ is less than $\alpha
d_{n-1}$ for some fixed constant $0<\alpha<1$. Here we assume that
$d_{n-1}\ll d_n$ for all $n$.  In the construction of
\cite{GaoJackson}, each $R\in \sR_n$ is obtained from a true rectangle
$R'$ by modifying the boundary in $n-1$ stages. At stage $k$ the
adjustments are on the order of scale $d_k$.  Thus, the boundaries of
the $R\in \sR_n$ become increasingly fractal-like as $n$ increases.
The key property the marker regions of $\sR_n$ have is that for any $x
\in F(2^{\Z^2})$, we have that $\rho(x, \partial \sR_n) \to \infty$,
where $\partial \sR_n$ denotes the union of the boundaries of the
regions $R$ in $\sR_n$. This construction, which results in marker
regions with the above {\em vanishing boundary property}, is the main
ingredient of the hyperfiniteness proof of \cite{GaoJackson}.

Theorem~\ref{dtm} gives us some additional information about this
construction. Namely, for $x \in F(2^{\Z^2})$, let $\varphi_x \colon
\N \to \N$ be given by $\varphi_x(n)=\rho(x, \partial \sR_n)$. The
orthogonal marker construction says that each $\varphi_x$ tends to
infinity with $n$ whereas Theorem~\ref{dtm} says that the growth rate
of the $\varphi_x$ can be arbitrarily slow.  More precisely, given any
$f \colon \N \to \N$ with $\limsup_n f(n)=+\infty$, there is an $x \in
F(2^{\Z^2})$ such that $\varphi_x(n)<f(n)$ for infinitely many
$n$. Thus, we cannot improve the orthogonal marker theorem by
prescribing a growth rate for the functions $\varphi_x$, not even if
we seek Borel, instead of clopen, finite subequivalence relations.

It is natural also to ask whether the fractal-like nature of the
$\sR_n$ is also necessary.  Could we have a sequence $\sR_n$ of Borel
finite subequivalence relations, with vanishing boundary, where the
regions $R\in \sR_n$ have a regular geometry? For instance, can we
have all marker regions in $\sR_n$ to be rectangles with 
edge lengths between $v(n)$ and $w(n)$ where $\lim_n v(n)=\infty$. 
The next result shows that
this potential improvement to the orthogonal marker construction is
also impossible. For simplicity we state the result only for
rectangles, but the proof works for reasonably regular polygons.

\begin{thm} \label{thm:nrb}
There does not exist a sequence $\sR_n$ of Borel finite subequivalence
relations on $F(2^{\Z^2})$ satisfying all the following:

\begin{enumerate}
\item[\rm (1)] {\em (regular shape)} For each $n$, each marker region
$R$ of $\sR_n$ is a rectangle.

\item[\rm (2)] {\em (bounded size)} \label{nrb2} 
For each $n$, there is an upper bound $w(n)$ on the size
of the edge lengths of the marker regions $R$ in $\sR_n$.

\item[\rm (3)] {\em (increasing size)} Letting $v(n)$ denote the
smallest edge length of a marker region $R$ of $\sR_n$, we have $\lim_n
v(n)=+\infty$.

\item[\rm (4)] {\rm (vanishing boundary)} \label{nrb4} For each $x \in
F(2^{\Z^2})$ we have that $\lim_n \rho(x, \partial \sR_n)=+\infty$.
\end{enumerate}
\end{thm}

\begin{proof}
Assume $\sR_n$ were Borel finite subequivalence relations with all the
stated properties.  We view conditions $p$ for our Cohen forcing
$\bbP=\bbP_{\Z^2}(2)$ as being partial functions $p \colon [a,b]
\times [c,d] \to \{ 0,1\}$ for some rectangle $A=[a,b]\times [c,d]$ in
$\Z^2$.  We will produce an $x \in F(2^{\Z^2})$ which will lie on the
boundary of a rectangle $R$ of $\sR_n$ for infinitely many $n$. This
will contradict property~(4).  Given $p \in \bbP$ and $k \in \N$, we
will produce a $q \leq p$ and an $n > k$ such that $q \forces (\xd_G
\in \partial \sR_n)$. Any $\bbP$-generic $x$ will then be as desired.

So, fix $p \in \bbP$ and $k \in \N$. We may assume that
$A=\dom(p)=[-a,a]\times [a,a]$.  By (3) we may choose $n>k$ large
enough so that the minimum edge length $v(n)$ of any $R$ of $\sR_n$ is
greater than $2(2a+1)^2$. Let $w=w(n)$ be the maximum edge length of any
$R$ of $\sR_n $, which is well-defined by (2). Let $b> 3w$, and let
$B=[0,b]\times [0,b] \subseteq \Z^2$.  Let $r \in \bbP$ be the
condition with domain $B$ obtained by restricting to $B$ the following
function $r'$:


\begin{align*}
r'(i,j) &= p(i' -a, j'-a), \\
\intertext{where}
i' & = i \mod(2a+1) \\
j'&=j+\displaystyle\frac{i-i'}{2a+1}  \mod (2a+1).
\end{align*}

\noindent
The function $r'$ is obtained by tiling $\Z^2$ with copies of $p$ as
follows. First put a copy of $p$ at $[0,2a]\times [0,2a]$ and
vertically stack copies of $p$ to tile the column
$[0,2a]\times\Z$. Then, on the vertical column $[2a+1,4a+1] \times \Z$
immediately to the right we shift this stack down by one. We continue
this, so on the vertical stack which is $c$ columns to the right, that
is, on $[c(2a+1), c(2a+1)+2a] \times \Z$, we shift down by $c\!\!\mod
(2a+1)$. This is illustrated in Figure~\ref{fig:nrb}.

\begin{figure}[ht]
\centering
\begin{tikzpicture}[scale=0.05]
\pgfmathsetmacro{\s}{2};
\draw (-10,0) to (110,0);
\draw (0,-10) to (0,110);
\draw (0,0) rectangle (100,100);

\foreach \j in {0,...,4}
{
\foreach \i in {1,...,9} {\draw (10*\j, 10*\i-\j*\s) rectangle (10*\j+10,10*\i+10-\j*\s);}
\draw (10*\j,0) rectangle (10*\j+10,10-\j*\s);
\draw (10*\j,100-\j*\s) rectangle (10*\j+10,100);
}

\foreach \j in {0,...,4}
{
\foreach \i in {1,...,9} {\draw (50+10*\j, 10*\i-\j*\s) rectangle (50+10*\j+10,10*\i+10-\j*\s);}
\draw (50+10*\j,0) rectangle (50+10*\j+10,10-\j*\s);
\draw (50+10*\j,100-\j*\s) rectangle (50+10*\j+10,100);
}

\end{tikzpicture}
\caption{The construction of the condition $r$.} \label{fig:nrb}
\end{figure}
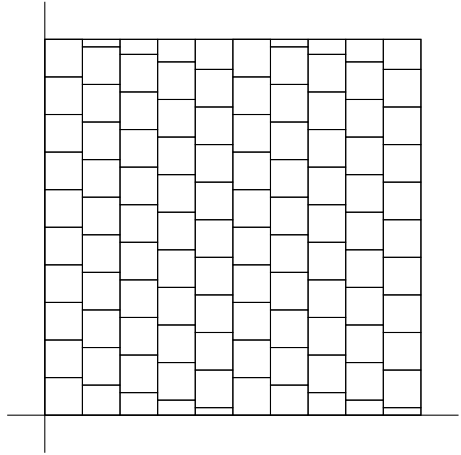

Let $x_G$ be a $\bbP$-generic real extending $r$.  Let $\sR_n(x_G)$
denote the set of rectangles in $\Z^2$ induced by the subequivalence
relation $\sR_n$ on the class $[x_G]$.  Since $b>3w$, there is a
rectangle $R$ in $\sR_n(x_G)$ lying entirely in $B=[0,b]\times
[0,b]$. Let $e$ be a horizontal edge of $R$. Then $e$ also lies
entirely in $B$. Since $e$ has length greater than $2(2a+1)^2$, the
offsets of the columns in the definition of $r$ gives that there is a
translate $A'=(i_0,j_0)+([-a,a] \times [-a,a])$ of $\dom(p)$ such that
$r \res A'= p$ and such that $e$ passes through the center point
$(i_0,j_0)$ of $A'$. By genericity, there is a condition $s \leq r$ in
$G$ such that
\[
s \forces (\exists R \in \sR_n(\xd_G) \ (i_0,j_0) \in \partial R )
\wedge (\xd_G \res A'=p)
\]

Let $\pi$ be the automorphism of $\bbP$ obtained by translating by
$(-i_0,-j_0)$.  Let $q=\pi(s)$. Then

\[
q \forces (\exists R \in \sR_n(\pi(\xd_G)) \ (i_0,j_0) \in \partial R)
\wedge (\pi(\xd_G) \res A'=p)
\]

\noindent
From the invariance of the $\sR_n$ this gives

\[
q\forces (\exists R \in \sR_n(\xd_G) \ (0,0) \in \partial R) \wedge
(\xd_G \res A =p)
\]

\noindent
So, $q\leq p$ and $q \forces \xd_G \in \partial \sR_n$.
\end{proof}

\section{Minimal \texorpdfstring{$2$}{2}-coloring Forcing \texorpdfstring{on $F(2^{\Z^d})$}{}}

For any countable group $\Gamma$, there is an $x \in F(2^\Gamma)$
which is a minimal $2$-coloring (cf. \cite{GJS1}, \cite{GJS2}). With
the corresponding orbit forcing $\bbP_x$, a generic real $x_G$
continues to be a minimal $2$-coloring in the generic extension
$V[G]$.  On the other hand, it is possible to directly define a
forcing notion which generically adds minimal $2$-colorings in
$2^\Gamma$. The advantage of this approach is that we will be able to
take advantage of some extra properties of the generic reals that are
not obviously available from a general $2$-coloring $x$ without other
features. In this section we will describe such a forcing notion,
which we call {\it minimal $2$-coloring forcing}, and some variations
of it, and use these forcing notions to prove some new results about
Borel complete sections.

It is possible to define the minimal $2$-coloring forcing for a
general group $\Gamma$, but such a definition is cumbersome to
describe, as it necessarily embodies the construction of a
$2$-coloring, which is not easy. For the case $\Gamma=\Z^d$, this
forcing has a simple and natural definition. Since our remaining
applications concern $F(2^{\Z^d})$, we will present the definition of
the minimal $2$-coloring forcing in this case. It will be clear that a
part of the definition is designed for producing a generic
$2$-coloring, and another part for producing a generic minimal
element. In applications, we sometimes need just the $2$-coloring
property of the generic real, sometimes we need just the minimality,
and frequently we need some additional properties which requires us to
modify the forcing.

We begin with a description of the basic forcing $\bbP_{mt}$ for
adding a minimal $2$-coloring in $2^{\Z^d}$. Since the definition is
essentially the same for all $\Z^d$ for $d\geq 1$, to ease notation we
consider the case $d=2$.

\begin{defn} \label{def:mt}
The {\em basic minimal $2$-coloring forcing} $\bbP_{mt}$ on $\Z^2$
consists of conditions
$$ \mathfrak{p}=(p,n,t_1,\dots, t_n,T_1,\dots, T_n, m, f_1,\dots, f_m,
F_1,\dots, F_m) $$ where
\begin{enumerate}
\item $p\in 2^{<\Z^2}$ with $\mbox{dom}(p)=[a, b]\times [c,d]$ for
  some $a<b$, $c<d$, $a, b,c,d\in \Z$;
\item $n, m\in \N$;
\item $t_1,\dots, t_n\in \Z^2-\{(0,0)\}$;
\item $f_1,\dots, f_m\in 2^{<\Z^2}$;
\item $T_1,\dots, T_n$, $F_1,\dots, F_m$ are finite subsets of $\Z^2$;
\end{enumerate}
such that the following conditions are satisfied:
\begin{enumerate}
\item[(a)] (2-coloring property) For any $1\leq i\leq n$ and
  $g\in\mbox{dom}(p)$ there is $\tau\in T_i$ such that $g+\tau,
  g+t_i+\tau\in\mbox{dom}(p)$ and $p(g+\tau)\neq p(g+t_i+\tau)$;
\item[(b1)] (minimality) For any $1\leq j\leq m$ and
  $g\in\mbox{dom}(p)$ there is $\sigma\in F_j$ such that
  $g+\sigma+\mbox{dom}(f_j)\subseteq\mbox{dom}(p)$ and for all $u\in
  \mbox{dom}(f_j)$, $p(g+\sigma+u)=f_j(u)$.
\item[(b2)] (minimality with flips) For any $1\leq j\leq m$ and
  $g\in\mbox{dom}(p)$ there is $\sigma\in F_j$ such that
  $g+\sigma+\mbox{dom}(f_j)\subseteq\mbox{dom}(p)$ and for all $u\in
  \mbox{dom}(f_j)$, $p(g+\sigma+u)=1-f_j(u)$.
\end{enumerate}

When we need to speak of another forcing condition $\mathfrak{q}$ we
will denote the main part of the forcing condition, namely the partial
function, as $q$, and the rest of the terms in the tuple as
$n(\mathfrak{q}), \vec{t}(\mathfrak{q}), \vec{T}(\mathfrak{q}),
m(\mathfrak{q}), \vec{f}(\mathfrak{q})$, and $\vec{F}(\mathfrak{q})$,
respectively.

If $\mathfrak{p},\mathfrak{q}\in\mathbb{P}$, we define the extension
relation $\mathfrak{p}\leq\mathfrak{q}$ by
\begin{enumerate}[label=(\roman*)]
\item $p\supseteq q$,
\item $n(\mathfrak{p})\geq n(\mathfrak{q})$,
\item for all $1\leq i\leq n(\mathfrak{q})$,
  $t_i(\mathfrak{p})=t_i(\mathfrak{q})$ and
  $T_i(\mathfrak{p})=T_i(\mathfrak{q})$,
\item $m(\mathfrak{p})\geq m(\mathfrak{q})$,
\item for all $1\leq j\leq m(\mathfrak{q})$,
  $f_j(\mathfrak{p})=f_j(\mathfrak{q})$ and
  $F_j(\mathfrak{p})=F_j(\mathfrak{q})$.
\end{enumerate}
\end{defn}

For $h\in \Z^2$ and $p\in 2^{<\Z^2}$, we define $h\cdot p$ by letting
$\dom(h\cdot p)=h + \dom(p)$ and for all $g\in\dom(p)$, $(h\cdot
p)(h+g)=p(g)$. For
$$ \mathfrak{p}=(p,n,t_1,\dots, t_n,T_1,\dots, T_n, m, f_1,\dots, f_m,
F_1,\dots, F_m), $$ define
$$ h\cdot \mathfrak{p}=(h\cdot p,n,t_1,\dots, t_n,T_1,\dots, T_n, m,
f_1,\dots, f_m, F_1,\dots, F_m). $$ The invariance of the forcing
notion is easy to check.

We now prove a few simple lemmas which guarantee that $\bbP_{mt}$ does
indeed add a minimal $2$-coloring in $2^{\Z^2}$.

\begin{lemma}\label{lem1}
For any $g\in \Z^2$, the set $D_g=\{\mathfrak{p}\in\bbP_{mt} \colon
g\in\dom(p)\}$ is dense in $\bbP_{mt}$.
\end{lemma}

\begin{proof}
Let $\mathfrak{q}\in\bbP_{mt}$ and $g\in \Z^2$. We need to find
$\mathfrak{p}\leq \mathfrak{q}$ with $g\in\dom(p)$.  Intuitively, we
use $q$ as a ``building block" and construct $p$ by a ``tiling" of $q$
until the domain of $p$ covers the element $g$.  More precisely,
suppose $\dom(q)=[a,b] \times [c,d]$.  If $g=(g_1,g_2)\in [a,b]\times
[c,d]$ we just take $\mathfrak{p}=\mathfrak{q}$. Without loss of
generality assume $g_1>b$, $g_2>d$ (the other cases being
similar). Let $w=b-a+1$, $h=d-c+1$.  Let
$$i_0=\displaystyle\left\lfloor
\displaystyle\frac{g_1-a}{w}\displaystyle\right\rfloor, \ \ j_0=\left\lfloor
\displaystyle\frac{g_2-c}{h}\displaystyle\right\rfloor. $$ We define $p$
with $\dom(p)=[a,b+i_0 w]\times[c,d+j_0 h]$ by letting $p(a+iw+i',
c+jh+j')=q(a+i',c+j')$ for all $0\leq i\leq i_0$, $0 \leq j \leq j_0$,
$0 \leq i' <w$, and $0\leq j'<h$.

Then define $n(\mathfrak{p})=n(\mathfrak{q})$,
$\vec{t}(\mathfrak{p})=\vec{t}(\mathfrak{q})$,
$\vec{T}(\mathfrak{p})=\vec{T}(\mathfrak{q})$,
$m(\mathfrak{p})=m(\mathfrak{q})$,
$\vec{f}(\mathfrak{p})=\vec{f}(\mathfrak{q})$, and
$\vec{F}(\mathfrak{p})=\vec{F}(\mathfrak{q})$.

We need to verify that $\mathfrak{p}\in\bbP_{mt}$. It suffices to
verify the conditions (a), (b1) and (b2). For (a) let $0\leq i\leq
n(\mathfrak{p})$ and let $h=(h_1,h_2)\in \dom(p)$. Let $k=(k_1,k_2)\in
[a,b]\times [c,d]$ be the unique element such that $w\mid (h_1-k_1)$,
$h\mid (h_2-k_2)$.  Let $\tau\in T_i$ be such that $k+\tau,
k+t_i+\tau\in [a,b]\times[c,d]$ and $q(k+\tau)\neq
q(k+t_i+\tau)$. Then by commutativity we have $h+\tau, h+t_i+\tau\in
\dom(p)$ and $$p(h+\tau)=q(k+\tau)\neq q(k+t_i+\tau)=p(h+t_i+\tau).$$
For (b1) let $1\leq j\leq m(\mathfrak{p})$ and
$h\in\mbox{dom}(p)$. Again let $k\in[a,b]\times [c,d]$ be the unique
element such that $w\mid(h_1-k_1)$, $h \mid (h_2-k_2)$. Then there is
$\sigma\in F_j$ such that $k+\sigma+\dom(f_j)\subseteq [a,b]\times
[c,d]$ and for all $u\in \dom(f_j)$, $q(k+\sigma+u)=f_j(u)$. Now
$h+\sigma+\dom(f_j)\subseteq \dom(p)$ and for all $u\in \dom(f_j)$,
$p(h+\sigma+u)=q(k+\sigma+u)=f_j(u)$. This proves (b1). The proof of
(b2) is similar.
\end{proof}

Note that in the construction of $p$ from $q$ in the above proof, $p$
is a ``tiling" by $q$. Denote $\overline{q}=1-q$ and call it the {\it
  flip} of $q$. Then $p$ could also be constructed as a ``tiling" by
both $q$ and $\overline{q}$, using an arbitrary combination of these
two kinds of ``tiles". The use of the flip tile was not necessary in
the above proof, but will be necessary in the following lemmas.

\begin{lemma}\label{lem2}
For any $t\in\Z^2-\{(0,0)\}$ the
set $$E_t=\{\mathfrak{p}\in\bbP_{mt}\colon \exists 1\leq i\leq
n(\mathfrak{p})\ t_i(\mathfrak{p})=t\}$$ is dense in $\bbP_{mt}$.
\end{lemma}

\begin{proof}
Let $t=(t_1,t_2)\in \Z^2-\{(0,0)\}$ and $\mathfrak{q}\in\bbP_{mt}$. We
will find $\mathfrak{p}\leq\mathfrak{q}$ with $t=t_i(\mathfrak{p})$
for some $1\leq i\leq n(\mathfrak{p})$. If $t=t_i(\mathfrak{q})$ for
some $1\leq i\leq n(\mathfrak{q})$ we just take
$\mathfrak{p}=\mathfrak{q}$. Otherwise, we will define $\mathfrak{p}$
so that $n(\mathfrak{p})=n(\mathfrak{q})+1$ and
$t=t_{n(\mathfrak{p})}(\mathfrak{p})$. For notational simplicity let
$n(\mathfrak{q})=n-1$ and assume $\dom(q)=[a,b]\times[c,d]$. Also,
assume $t_1,t_2>0$ (the other cases being similar).  Let $w=b-a+1$,
$h=d-c+1$, and let $b+t_1=a+i_0 w+i_1$, $d+t_2=c+j_0 h+j_1$ where
$0\leq i_1 <w$, $0 \leq j_1 <h$. Note that at least one of $i_0,j_0$
is greater than $0$.  Define $p$ with $\dom(p)=[a,b+i_0w]\times
[c,d+j_0h]$ in a similar fashion as we did in the proof of the
previous lemma. Specifically, let $p(a+iw+i',c+jh+j')=q(a+i', c+j')$
for all $0\leq i\leq i_0$, $0 \leq j \leq j_0$, $0\leq i'<w$, and $0
\leq j' <h$, except in the case $i=i_0$ and $j=j_0$. If $q(a + i_1,c + j_1)\neq q(b,d)$,
then let $p(a+i_0w+i',c+j_0h+j')=q(a+i',c+j')$ for all $0\leq i' < w$,
$0\leq j' < h$.  Otherwise, if $q(a + i_1, c + j_1)=q(b,d)$, then let
$p(a+i_0w+i', c+j_0h+j')=1-q(a+i',c+j')$ for all such $i',
j'$. Intuitively, $p$ is a tiling of $q$ and $\overline{q}$, with only
the last tile being possibly $\overline{q}$. The choice between $q$
and $\overline{q}$ is made to ensure $p(b,d)\neq p((b,d)+t)$.

Let $T_n=[-i_0w, b-a]\times [-j_0h,d-c]$. Then
$n(\mathfrak{p})=n=n(\mathfrak{q})+1$,
$\vec{t}(\mathfrak{p})=\vec{t}(\mathfrak{q})^\smallfrown t$,
$\vec{T}(\mathfrak{p})=\vec{T}(\mathfrak{q})^\smallfrown T_n$,
$m(\mathfrak{p})=m(\mathfrak{q})$,
$\vec{f}(\mathfrak{p})=\vec{f}(\mathfrak{q})$, and
$\vec{F}(\mathfrak{p})=\vec{F}(\mathfrak{q})$.

The proof of (a) for all elements of $\vec{t}(\mathfrak{q})$ is the
same as in the previous proof. We verify (a) only for $t$. Let $g\in
\dom(p)$. Let $\tau=(b,d)-g$. Then $\tau\in T_n$, $g+\tau=(b,d)$ and
$g+t+\tau=(b,d)+t$. Both $g+\tau, g+t+\tau\in\dom(p)$ and
$p(g+\tau)=p(b,d)\neq p((b,d)+t)=p(g+t+\tau)$. Thus (a) holds.

We verify (b1) and (b2). For (b1) let $1\leq j\leq m(\mathfrak{p})$ and
$g=(g_1,g_2)\in \dom(p)$. If
the building block containing $g$ is a copy of $q$, then the proof is
the same as in the proof of Lemma \ref{lem1}. So suppose that the block
$(i \cdot w, j \cdot h) + [a, b] \times [c, d]$ containing $g$ is a
copy of $\overline{q}$ (by our construction, this necessitates $i = i_0$
and $j = j_0$). Let $k=(k_1,k_2)\in [a,b]\times[c,d]$ be the
unique element such that $w\mid (g_1-k_1)$, $h \mid (g_2-k_2)$. Then
by (b2) for $\mathfrak{q}$ there is $\sigma\in F_j$ such that
$k+\sigma+\dom(f_j)\subseteq [a,b]\times [c,d]$ and for all $u\in
\dom(f_j)$, $q(k+\sigma+u)=1-f_j(u)$. Now $g+\sigma+\dom(f_j)\subseteq
\dom(p)$ and for all $u\in \dom(f_j)$,
$p(g+\sigma+u)=1-q(k+\sigma+u)=f_j(u)$. This proves (b1) for
$\mathfrak{p}$. The proof of (b2) is similar.
\end{proof}

\begin{lemma}\label{lem3}
For any finite set $A\subseteq \Z^2$, the
set $$D_A=\{\mathfrak{p}\in\bbP_{mt} \colon \exists 1\leq j\leq
m(\mathfrak{p})\ A\subseteq \mbox{\rm dom}(f_j(\mathfrak{p}))\}$$ is
dense in $\bbP_{mt}$.
\end{lemma}

\begin{proof}
Let $\mathfrak{q}\in\mathbb{P}$ and $A\subseteq \Z^2$ be finite. By
repeated application of Lemma~\ref{lem1} we can obtain
$\mathfrak{r}\leq\mathfrak{q}$ such that $A\subseteq\dom(r)$.  Let
$\dom(r)=[a,b]\times [c,d]$. We define $p$ with
$\dom(p)=[a,2b-a+1]\times [c,d]$ to be a copy of $r$ (on $[a,b]\times
[c,d]$) and a copy of $\overline{r}$ immediately to the right (on
$[b+1,2b-a+1]\times [c,d]$).

Then define $n(\mathfrak{p})=n(\mathfrak{r})$,
$\vec{T}(\mathfrak{p})=\vec{T}(\mathfrak{r})$,
$m(\mathfrak{p})=m(\mathfrak{r})+1$,
$\vec{f}(\mathfrak{p})=\vec{f}(\mathfrak{r})^\smallfrown r$, and
$\vec{F}(\mathfrak{p})=\vec{F}(\mathfrak{r})^\smallfrown ([a-2b-1,
  b-2a+1]\times [-d,-c])$.

It suffices to verify that $\mathfrak{p}\in\bbP_{mt}$. (a) holds by a
similar argument as in the proofs of previous two lemmas. For (b1) and
(b2), the proof for $1\leq j\leq m(\mathfrak{r})$ is similar to the
proofs in the previous two lemmas. Finally, for $g=(g_1,g_2)\in
\dom(p)$ there are obviously $\sigma, \sigma'\in [a-2b-1,
  b-2a+1]\times [-d,-c]$ such that $g+\sigma=(0,0)$ and
$g+\sigma'=(b-a+1,0)$. Then for all $u\in [a,b]\times [c,d]$,
$p(g+\sigma+u)=r(u)$ and $p(g+\sigma'+u)=1-r(u)$.
\end{proof}

In fact, the above proof gives the following lemma.

\begin{lemma}\label{lem3.5}
For any $\mathfrak{q}\in\bbP_{mt}$, the set
$$ D_q=\{\mathfrak{p}\in\bbP_{mt}\colon \exists 1\leq j\leq
m(\mathfrak{p})\ q\subseteq f_j(\mathfrak{p})\} $$ is dense below
$\mathfrak{q}$ in $\bbP_{mt}$.
\end{lemma}

\begin{proof}
As in the previous proof, define $p$ to be a tiling with one copy of
$q$ and a copy of $\overline{q}$ to the right.
\end{proof}

Putting these lemmas together we have the following.

\begin{lemma} \label{mt4}
If $x_G$ is generic for $\bbP_{mt}$, then $x_G$ is a minimal
$2$-coloring.
\end{lemma}

\begin{proof}
Lemma~\ref{lem1} gives that $x_G \in 2^{\Z^2}$. Lemma~\ref{lem2} gives
that $x_G$ is a $2$-coloring. To see that $x_G$ is minimal, let $A
\subseteq \Z^2$ be finite, and let $f=x_G\res A$. Let $\mathfrak{q}
\in G$ be such that $\dom(q) \supseteq A$. From Lemma~\ref{lem3.5}
there is a $\mathfrak{p} \in G$ with $f \subseteq f_j$ for some $1
\leq j \leq m(\mathfrak{p})$. We then have that $F_j(\mathfrak{p})$
witnesses the minimality condition for $A$, that is, for all $g \in
\Z^2$ there is a $t \in F_j(\mathfrak{p})$ such that $x_G(g+t
+u)=f(u)$ for all $u\in \dom(f)=A$.
\end{proof}

Using Lemma~\ref{mt4} and the proof of Theorem~\ref{dtm} we can get a
direct proof of Corollary~\ref{cor:dtm} which is self-contained and
does not rely on the {\em a priori} construction of a minimal $2$-coloring.

We next consider a relatively minor variation of $\bbP_{mt}$ which
will turn out to have interesting consequences.

\begin{defn}
The {\it odd minimal $2$-coloring forcing} $\bbP_{omt}$ is defined
exactly as $\bbP_{mt}$ in Definition~\ref{def:mt} except that we add
the requirement that if $\dom(p)=[a,b]\times [c,d]$ then both $b-a+1$
and $d-c+1$ are odd.  That is, the rectangle representing the domain
of $p$ must have odd numbers of vertices on each of the sides.
\end{defn}

The following result intuitively states that any Borel complete
section has an odd recurrence on some orbit.

\begin{thm} \label{thm:oddrecurrence}
Let $O=\{ g\in \Z^2\,:\, \mbox{$\|g\|$ is odd }\}$. If $B \subseteq
F(2^{\Z^2})$ is a Borel complete section then there is $x\in
F(2^{\Z^2})$ and finite $T\subseteq O$ such that for any $y\in[x]$,
$T\cdot y\cap B\neq\varnothing$.
\end{thm}

\begin{proof}
Let $x_G$ be a generic real for $\bbP_{omt}$. Since $B^{V[G]}$
continues to be a Borel complete section, there is $g_0\in \Z^2$ such
that $g_0\cdot x_G\in B$. So there is $\mathfrak{p}\in G$ such that
$\mathfrak{p}\forces g_0\cdot\dot{x}_G\in B$. Now note that for any
$\mathfrak{q} \leq \mathfrak{p}$ there is an $\mathfrak{r} \leq
\mathfrak{q}$ such that $r$ contains two disjoint copies of $p$ at an
odd distance apart. We can, in fact, get $r$ by placing two copies of
$q$ next to each other, since $\dom(q)$ is a rectangle with an odd
number of vertices on each side. This implies that there is a
$\mathfrak{q}\in G$ with $\mathfrak{q}\leq \mathfrak{p}$ and such that
$q$ contains two disjoint copies of $p$ an odd distance apart. Let
$\|q\|$ denote the sum of the side lengths of $\dom(q)$.

Since $x_G$ is minimal, there is $N\in \N$ such that for all $g\in
\Z^2$ there is a $\tau\in\Z^2$ with $\|\tau\|\leq N$ such that
$\tau\cdot (g\cdot x_G)\in U_q$, where $U_q$ is the basic open set in
$2^{\Z^2}$ determined by $q$. Now let $T$ be all the elements $g\in O$
with $\|g\|\leq N+\|q\|+\|g_0\|$. Fix any $y\in[x_G]$. Fix $\tau$ with
$\|\tau\|\leq N$ such that $\tau\cdot y\in U_q$. In particular
$\tau\cdot y\in U_p$. Let $h\in O$ be such that $\|h\|\leq\|q\|$ and
$h\cdot (\tau\cdot y)\in U_p$. Since both $\tau\cdot y$ and $h\cdot
(\tau\cdot y)$ are also generic and extend the condition
$\mathfrak{p}$, we have that $g_0\cdot (\tau\cdot y)\in B$ and
$g_0\cdot (h\cdot (\tau\cdot y))\in B$. We have both $\|\tau+g_0\|\leq
\|\tau\|+\|g_0\|\leq N+\|g_0\|$ and $\|\tau+h+g_0\|\leq
N+\|q\|+\|g_0\|$. Since $h\in O$, one of $\tau+g_0$ and $\tau+h+g_0$
is an element of $O$, and therefore an element of $T$. Thus we have
shown that $T\cdot y\cap B\neq\varnothing$ as required.
\end{proof}

The next theorem is about Borel chromatic $k$-colorings of
$F(2^{\Z^2})$. Intuitively, it states that there does not exists a
Borel chromatic $k$-coloring $f$ of $F(2^{\Z^2})$ such that on every
orbit there are arbitrarily large regions on which $f$ induces a
chromatic $2$-coloring.

\begin{thm} \label{ltc}
Suppose $f \colon F(2^{\Z^2}) \to \{0,1,\dots,k-1\}$ is a Borel
function.  Then there is an $x \in F(2^{\Z^2})$ and an $M \in \N$ such
that if the map $t \mapsto f(t \cdot x)$ is a chromatic $2$-coloring
on $[a,b]\times [c,d]$, then $b-a, d-c \leq M$.
\end{thm}

\begin{proof}
Fix a Borel function $f \colon F(2^{\Z^2})\to \{ 0,1, \dots,k-1\}$.
Let $x_G$ be a generic real for $\bbP_{omt}$. We claim that $x_G$ is
as required.  Recall that $x_G$ is a minimal $2$-coloring. Suppose
that on $x_G$ the function $f$ had arbitrarily large regions which
were chromatic $2$-colorings.  Let $\mathfrak{p}\in G$ be such that
$\mathfrak{p} \forces (f(\xd_G)=i)$ for some fixed $i \in \{
0,\dots,k-1\}$.  Let $\mathfrak{q}\in G$ with $\mathfrak{q}\leq
\mathfrak{p}$ and such that $q$ contains two disjoint copies of $p$ an
odd distance apart. Let again $\|q\|$ denote the sum of the side
lengths of $\dom(q)$.

By minimality of $x_G$, let $N \in \N$ be such that for all $g \in
\Z^2$, there is a $\tau\in \Z^2$ with $\| \tau \| \leq N$ such that
$\tau\cdot (g \cdot x_G) \in U_q$. Since $x_G$ is assumed to have
arbitrarily large regions which are chromatically $2$-colored by $f$,
fix $\sigma \in \Z^2$ such that $\tau \mapsto f(\tau\cdot (\sigma
\cdot x_G))$ is a chromatic $2$-coloring of a square $[-a,a]^2$ in
$\Z^2$ with $a > N+ \| q\|$. Fix $\tau$ with $\| \tau \| \leq N$ such
that $\tau\cdot (\sigma\cdot x_G) \in U_q$. In particular $\tau\cdot
(\sigma \cdot x_G) \in U_p$.  Let $h \in O$ be such that $\|h\|\leq
\|q\|$ and $h\cdot (\tau\cdot (\sigma\cdot x_G))\in U_p$.  Now both
$\tau\cdot (\sigma\cdot x_G)$ and $h\cdot (\tau\cdot (\sigma\cdot
x_G))$ are generic and both extend the condition $\mathfrak{p}$. So,
$f(\tau\cdot (\sigma\cdot x_G))=i=f(h\cdot (\tau\cdot (\sigma \cdot
x_G)))$. This is a contradiction as $\| \tau\|\leq N\leq a$ and $\| h
+ \tau\|\leq N+ \|h\| \leq a$, and $\tau$, $h + \tau$ are an odd
distance apart in $\Z^2$.
\end{proof}

\begin{rem}It is easy to construct Borel $f \colon F(2^{\Z^2})\to \{ 0,1\}$ such that for comeager many 
$x \in F(2^{\Z^2})$ we have that $t \mapsto f(t \cdot x)$ has
  arbitrarily large regions which are chromatically $2$-colored.  We
  can, in fact, take $f(x)=x(0,0)$.  Similarly, measure-one (in the
  natural product measure) many $x \in F(2^{\Z^2})$ have arbitrarily
  large regions which are chromatically $2$-colored by $f$. This shows
  that ordinary category arguments (or, equivalently, Cohen forcing)
  and measure arguments are not sufficient to prove Theorem~\ref{ltc}.
\end{rem}

\section{Grid Periodicity Forcing} \label{sec:gp}
In this last section we introduce another variation of the minimal
$2$-coloring forcing which (similar to Theorem~\ref{ltc}) will show
that Borel complete sections in $F(2^{\Z^d})$ must have orbits on
which highly regular structure is exhibited.  We will take $d=2$ for
the following arguments for simplicity, though the arguments in the
general case are only notationally more complicated.

As before, we can describe our forcing either as a special case of
orbit forcing (by first building a particular minimal $2$-coloring $x$
and then considering $\bbP_x$), or we can describe the forcing
directly. Here we again describe the forcing directly.

\begin{defn} \label{def:gpf}
Let $n$ be a positive integer. The {\it grid periodicity forcing}
$\bbP_{gp}(n)$ is defined as follows.
\begin{enumerate}
\item
A condition $p \in \bbP_{gp}(n)$ is a function
$$p\colon R\setminus\{ u\} \to \{ 0,1\}$$ where $R=[a,b]\times [c,d]$
is a rectangle in $\Z^2$ with $w=b-a+1$, $h=d-c+1$ both powers of $n$,
and $u\in R$. We write $R(p)$, $w(p)$, $h(p)$, $u(p)$ for the
corresponding objects and parameters.
\item
The conditions are ordered by $p \leq q$ iff
\begin{enumerate}
\item \label{gpf2a} $R(p)$ is obtained by a rectangular tiling by
  $R(q)$, that is, $R(p)$ is the disjoint union $R(p)=\bigcup_{t \in
  A} t \cdot R(q)$ where $A$ is a subset of $\Z^2$ of the form
$$A=\left\{ (iw(q), jh(q))\,\colon\, i_0\leq i \leq i_1, j_0 \leq j
  \leq j_1\right\}$$ for some $i_0 \leq i_1,j_0 \leq j_1$;
\item \label{gpf2b} If $c \in \dom(q)$ and $t\in A$, then
  $p(c+t)=q(c)$;
\item
For some $t\in A$ we have $u(p)=u(q)+t$.
\end{enumerate}
\end{enumerate}
\end{defn}

\begin{figure}[ht]
\centering
\begin{tikzpicture}[scale=0.03]
\draw[step=20] (-60,-60) grid (60,60); 
\draw[fill, lightgray] (0,0) rectangle (20,20); 
\foreach \i in {-3,...,2}  \foreach \j in  {-3,...,2} { \draw[fill] (\i*20+15,\j*20+15) circle (5mm);}
\end{tikzpicture}
\caption{The extension relation in the grid periodicity forcing
  $\bbP_{gp}$.\label{fig:gp}}
\end{figure}
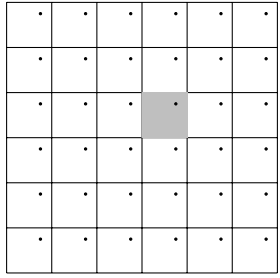
When $n$ is understood or does not matter, we will just write
$\bbP_{gp}$ instead of $\bbP_{gp}(n)$. The conditions and extension
relation for $\bbP_{gp}$ are illustrated in Figure~\ref{fig:gp}.  A
$\bbP_{gp}$-generic $G$ gives a real $x_G \in 2^{\Z^2}$. We next
establish the properties of this real.

\begin{lemma} \label{lem:gpmt}
$x_G$ is a minimal $2$-coloring.
\end{lemma}

\begin{proof}
Let $q \in \bbP_{gp}$ and let $s \in \Z^2\setminus\{ (0,0)\}$. Let $p
\leq q$ be defined as follows.  First extend $R(q)$ to $R(p)$ by
picking sufficiently large intervals $[i_0, i_1], [j_0, j_1]$ as in
Definition~\ref{def:gpf}(\ref{gpf2a}), and setting
$$A=\left\{ (iw(q), jh(q))\,\colon\, i_0\leq i \leq i_1, j_0 \leq j
\leq j_1\right\}$$ and $R(p)=\bigcup_{t\in A} t\cdot R(q)$, so that
$u(q)+s \in R(p)$. For all $c \in R(q)\setminus\{ u(q)\}$ and $t\in
A$, we have $p(c+t)=q(c)$ by Definition~\ref{def:gpf}
(\ref{gpf2b}). We then define $p$ on the point $u(q)$, and possibly
$u(q)+s$ if $p$ is not already defined there, so that $p(u(q)) \neq
p(u(q)+s)$. We finally extend the domain so that $p$ is a condition in
$\bbP_{gp}$; in other words, we pick a point of the form $u(q)+t$ for
some $t\in A$ to be the new point $u(p)$ at which $p$ is undefined,
and define $p$ arbitrarily elsewhere. What we have just shown is the
density of the set
$$ D_s=\{ p\in\bbP_{gp}\,:\, \exists g\in\dom(p)\ (g+s\in
\dom(p)\mbox{ and } p(g)\neq p(g+s))\}. $$

Let $p\in D_s\cap G$ be arbitrary. Fix $g_0\in\dom(p)$ so that
$g_0+s\in\dom(p)$ and $p(g_0)\neq p(g_0+s)$. Let
$T=[-w(p),w(p)]\times[-h(p),h(p)]$. We claim that $p$ forces the
$2$-coloring property for the shift $s$ with witnessing set $T$. That
is,
\[
p \forces \forall g \in \Z^2\ \exists t \in T\ (\xd_G(g+t)\neq
\xd_G(g+s+t))
\]
To see this, note that from the definition of the extension relation
we have that $x_G$ is a tiling by $p$ except at points of the form
$u(p)+(iw(p),jh(p))$ for $(i,j)\in\Z^2$.  That is, if $c \in \dom(p)$
then $x_G(c+(iw(p),jh(p)))=p(c)$ for all $(i,j) \in \Z^2$. Any $g \in
\Z^2$ is in $R(p)+(iw(p),jh(p))$ for some $(i,j)\in\Z^2$, and so there
is a $\tau \in T$ such that $g+\tau$ is of the form $g_0+(iw,jh)$. So
$x_G(g+\tau)=p(g_0)\neq p(g_0+s)=x_G(g+s+\tau)$. This shows that $x_G$
is a $2$-coloring.

To see that $x_G$ is minimal, fix a finite $F \subseteq \Z^2$. Note
that the set
$$ S_F=\{p\in \bbP_{gp}\,:\, F \subseteq \dom(p)\} $$ is dense. Let
$p\in S_F\cap G$, and again let $T=[-w(p),w(p)]\times[-h(p),h(p)]$.
Let $E$ be the set of all points of the form $u(p)+(iw(p),jh(p))$ for
$(i,j)\in\Z^2$. Then $F\cap E=\varnothing$. Since $x_G$ is a tiling by
$p$ off $E$, it follows that for any $g \in \Z^2$ there is a $\tau \in
T$ such that $p(\sigma)=x_G(g+\tau+\sigma)$ for all $\sigma \in
F$. Thus $x_G$ is minimal.
\end{proof}

Since $x_G$ is a minimal $2$-coloring, we certainly have that $x_G$ is
not periodic, that is, $x_G \in F(2^{\Z^2})$. However, $x_G$ satisfies
some weak form of periodicity as the next lemma shows.

\begin{lemma}
Let $x_G$ be a generic real for $\bbP_{gp}(n)$.
\begin{enumerate}
\item[\rm (i)] For any vertical
or horizontal line $\ell$ in $\Z^2$, $x_G \res \ell$ is periodic
with period a power of $n$.
\item[\rm (ii)] For any finite $A \subseteq \Z^2$,
there is a lattice $L = (w \Z) \times (h \Z)$, with both $w$ and $h$
powers of $n$, and there is $u \in \Z^2 \setminus (A + L)$
such that $x_G$ is constant on $k + L$ whenever $k + L \neq u + L$.
\end{enumerate}
\end{lemma}

\begin{proof}
(ii). The proof is similar to the proof that $x_G$ is minimal in
Lemma~\ref{lem:gpmt}.  Given $A$ and $q \in \bbP_{gp}$, there is a $p
\leq q$ with $A \subseteq R(p) \setminus \{u(p)\}$. There is such a $p$ which forces that
$x_G$ has horizontal and vertical periods $w(p)$ and $h(p)$ (which are
powers of $n$) off of the set $u(p) + L(p)$, where $L(p)=\{ (iw(p),jh(p))\, \colon\,
(i,j) \in \Z^2\}$. It follows by genericity that $x_G$ has the stated
grid periodicity property.

(i). Given any vertical or horizontal line $\ell$ in $\Z^2$, the set of
$p\in\bbP_{gp}$ with
$$ \ell\cap (u(p) + L(p))=\ell\cap \{u(p)+(iw(p),jh(p))\,:\, (i,j)\cap
\Z^2\}=\varnothing $$ is dense. This implies that $x_G\upharpoonright
\ell$ has a period $n^k$ for some $k$.
\end{proof}

As an application of grid periodicity forcing we now have the
following structure theorem for Borel complete sections of
$F(2^{\Z^2})$.

\begin{thm} \label{thm:gp}
Let $B \subseteq F(2^{\Z^2})$ be a Borel complete section. Then there
is an $x \in F(2^{\Z^2})$ and a lattice $L= k+\{ (iw,jh) \colon (i,j)
\in \Z^2\}$ such that $L\cdot x \subseteq B$.
\end{thm}

\begin{proof}
Let $x_G$ be a generic real for $\bbP_{gp}$. We claim that $B \cap
[x_G]$ contains a lattice as required. Since $B \cap [x_G]\neq
\varnothing$, we may fix $k \in \Z^2$ and $q \in G$ such that $q
\forces (k \cdot \xd_G \in B)$. For any $(i,j)\in\Z^2$, let
$\pi_{i,j}$ be the translation defined by
$\pi_{i,j}(g)=g+(iw(q),jh(q))$. Then $\pi_{i,j}$ induces an
automorphism of $\bbP_{gp}$ and
$$\pi_{i,j}(q) \forces (\pi_{i,j}(k) \cdot \xd_G \in B).$$ Note that
$\pi_{i,j}(k)\cdot \xd_G= (k+(iw(p),jh(p)))\cdot \xd_G$.  It suffices
therefore to show that $G$ contains the condition $\pi_{i,j}(q)$. By
density, there is a $p\leq q$ in $G$ with $R(p) \supseteq R(q) \cup
R(\pi_{i,j}(q))$. It is clear, however, from the definition of the
extension relation that $p\leq \pi_{i,j}(q)$. Thus, $\pi_{i,j}(q)\in
G$ as well.
\end{proof}

We mention that while A. Marks was visiting the authors, he used
forcing methods to generalize the above theorem to all countable
residually finite groups $\Gamma$ \cite{Ma}.

The proof of Theorem~\ref{thm:gp} also gives the following variation
of Theorem~\ref{thm:gp}.

\begin{thm} \label{thm:gp2}
Let $f \colon F(2^{\Z^2}) \to \N$ be Borel. Then there is an $x \in
F(2^{\Z^2})$ and a lattice $L \subseteq \Z^2$ such that the map $s
\mapsto f(s \cdot x)$ is constant on $L$.
\end{thm}

Considering the characteristic function of the Borel set $B$ gives:

\begin{cor}
If $B \subseteq F(2^{\Z^2})$ is Borel, then there is an $x \in
F(2^{\Z^2})$ such that either $\{ s \,\colon\, s\cdot x \in B\}$ or
$\{ s \,\colon\, s\cdot x \in 2^{\Z^2}\setminus B\}$ contains a
lattice $L$ in $\Z^2$.
\end{cor}

\end{document}